\documentclass[11pt]{amsart}
\numberwithin{equation}{section}
\usepackage{epsf, amsmath, amsfonts, amscd, amsthm, amssymb,tikz-cd}
\usepackage{xcolor}

\theoremstyle{definition}
\newtheorem{theorem}{Theorem}[section]
\newtheorem{lemma}[theorem]{Lemma}
\newtheorem{proposition}[theorem]{Proposition}

\newtheorem{definition}[theorem]{Definition}

\newtheorem{remark}[theorem]{Remark}
\newtheorem{lem}[theorem]{Lemma}

\newtheorem{notation}[theorem]{Notation}
\newtheorem{question}[theorem]{Question}

\newtheorem{example}[theorem]{Example}

\newcounter{TmpEnumi}
\newcounter{Tmp2}

\newcommand{\beq}{\begin{equation}}
\newcommand{\eeq}{\end{equation}}
\newcommand{\beqr}{\begin{eqnarray*}}
\newcommand{\eeqr}{\end{eqnarray*}}
\newcommand{\bal}{\begin{align*}}
\newcommand{\eal}{\end{align*}}
\newcommand{\bei}{\begin{itemize}}
\newcommand{\eei}{\end{itemize}}
\newcommand{\limi}[1]{\lim_{{#1} \to \infty}}

\newcommand{\af}{\alpha}
\newcommand{\bt}{\beta}
\newcommand{\gm}{\gamma}
\newcommand{\dt}{\delta}
\newcommand{\ep}{\varepsilon}
\newcommand{\zt}{\zeta}
\newcommand{\et}{\eta}
\newcommand{\ch}{\chi}
\newcommand{\io}{\iota}
\newcommand{\te}{\theta}
\newcommand{\ld}{\lambda}
\newcommand{\sm}{\sigma}
\newcommand{\kp}{\kappa}
\newcommand{\ph}{\varphi}
\newcommand{\ps}{\psi}
\newcommand{\rh}{\rho}
\newcommand{\om}{\omega}
\newcommand{\ta}{\tau}

\newcommand{\Dt}{\Delta}

\newcommand{\Z}{{\mathbb{Z}}}
\newcommand{\R}{{\mathbb{R}}}
\newcommand{\C}{{\mathbb{C}}}
\newcommand{\N}{{\mathbb{Z}}_{> 0}}
\newcommand{\Nz}{{\mathbb{Z}}_{\geq 0}}

\newcommand{\id}{{\mathrm{id}}}

\newcommand{\spec}{{\mathrm{sp}}}

\newcommand{\diag}{{\mathrm{diag}}}

\newcommand{\spn}{{\mathrm{span}}}
\newcommand{\card}{{\mathrm{card}}}
\newcommand{\Aut}{{\mathrm{Aut}}}
\newcommand{\Ad}{{\mathrm{Ad}}}

\newcommand{\Inn}{\operatorname{Inn}}
\newcommand{\Ct}{\operatorname{Ct}}

\newcommand{\Out}{{\mathrm{Out}}}
\newcommand{\T}{{\mathrm{T}}}

\newcommand{\andeqn}{\qquad {\mbox{and}} \qquad}

\newcommand{\wolog}{without loss of generality}

\newcommand{\ca}{C*-algebra}

\newcommand{\hm}{homomorphism}

\newcommand{\tst}{tracial state}

\newcommand{\pj}{projection}

\newcommand{\ct}{continuous}

\renewcommand{\S}{\subset}

\newcommand{\I}{\infty}

\begin{document}
%
%
% TO MAKE TITLE
%
\title[Opposite Algebras]{K-Theory and the Universal Coefficient Theorem
 for Simple Separable Exact C*-Algebras
 Not Isomorphic to Their Opposites}

\author{N.~Christopher Phillips and Maria Grazia Viola}

\address{\hskip-\parindent Department of Mathematics \\
 University of Oregon \\
 Eugene OR 97403-1222 \\
 USA}

\address{\hskip-\parindent
 Department of Mathematical Sciences \\
 Lakehead University\\
 500 University Avenue \\
 Orillia ON L3V 0B9 \\
 Canada}

\email{mviola@lakeheadu.ca}


\date{1~September 2022}

\thanks{N.~Christopher Phillips was partially supported
by NSF Grants DMS-0701076 and DMS-1101742,
and by the
Simons Foundation Collaboration Grant for Mathematicians \#587103.
Maria Grazia Viola was partially supported
by a Research Development Fund from Lakehead
University and an NSERC Discovery Grant.}

\begin{abstract}
We construct uncountably many mutually nonisomorphic
simple separable stably finite unital exact \ca{s} which are not
isomorphic to their opposite algebras.
In particular, we prove that
there are uncountably many possibilities for the $K_0$-group,
the $K_1$-group, and the tracial state space of such an algebra.
We show that these C*-algebras satisfy
the Universal Coefficient Theorem, which is new even for the
already known example of an exact C*-algebra
nonisomorphic to its opposite algebra produced in
earlier work.
\end{abstract}

\maketitle

\section{Introduction}\label{Sec_Intro}

In the last two decades there has been much interest in
finding examples of simple \ca{} not isomorphic
to their opposite algebras.
The motivation for this work
is the Elliott classification program for simple  \ca{s},
which shows that simple separable nuclear unital
\ca{s} that absorb the Jiang-Su algebra~$Z$
tensorially and satisfy the Universal Coefficient Theorem
are classified up to isomorphism by an invariant consisting
of $K$-theory and tracial information.
(See Corollary~D in~\cite{CETWW}, or combine Corollary 4.11
in~\cite{EGLN} and Theorem~A in~\cite{CETWW}.)
For a unital \ca{} $A$, the Elliott invariant is given by
\[
\mathrm{Ell} (A)
 = \bigl( K_0 (A), K_0 (A)_{+}, [1_A], K_1 (A), \T (A), \rho \bigr),
\]
where $(K_0 (A), K_0 (A)_{+}, [1_A], K_1 (A))$
is the scaled ordered
K-theory of~$A$, $\T (A)$ denotes the tracial state simplex, and
$\rho \colon K_0 (A) \times \T (A) \to \mathbb{R}$
is the natural pairing map, $\rho ([p]-[q], \tau) = \tau (p) - \tau (q)$
for projections $p, q \in M_{\I} (A)$ and $\ta \in \T (A)$.
Since the Elliott invariant of a \ca{} is the same as
that of its opposite algebra, simple \ca{s} which
are not isomorphic to their opposite algebra provide examples
of simple \ca{s} that are not isomorphic
despite having the same Elliott invariant.

In~\cite{PhV} we constructed an example of a simple separable
unital exact \ca~$A$ not isomorphic to its opposite algebra.
The algebra $A$ has a number of nice properties:
it is stably finite and approximately divisible,
and it has real rank zero, stable rank one,
and a unique tracial state.
The order on projections over~$A$ is determined by traces, and
$A$ tensorially absorbs the Jiang-Su algebra~$Z$.
Its K-theory is given by
$K_0 (A) \cong {\mathbb{Z}} \big[ \tfrac{1}{3} \big]$
and $K_1 (A) = 0$.
Its Cuntz semigroup is
$W (A)
 \cong {\mathbb{Z}} \big[ \tfrac{1}{3} \big]_{+} \amalg (0, \infty)$.

The purpose of this article is to exhibit many examples
of simple separable exact \ca{s} not isomorphic to
their opposite algebras, and to prove that they satisfy the
Universal Coefficient Theorem.
(This is new even for the example in~\cite{PhV}.)
In particular, we prove that there
are uncountably many possibilities for the K-theory of such
an algebra, while still preserving most of the good properties
of the algebra in~\cite{PhV}.
For any odd prime $q$ such that $- 1$ is not a square mod~$q$,
and for any UHF algebra $B$ stable under tensoring with
the $q^{\infty}$~UHF algebra
(the algebra ${\displaystyle{\bigotimes_{n = 1}^{\infty} M_q}}$),
we produce a simple separable exact \ca~$D$,
not isomorphic to its opposite algebra,
with real rank zero and a unique tracial state,
such that $K_* (D) \cong K_* (B)$.
For any $q$ and $B$ as above,
and for any Choquet simplex~$\Dt$,
we give a simple separable exact \ca~$D$,
not isomorphic to its opposite algebra,
with real rank one, such that $K_* (D) \cong K_* (B)$,
and whose tracial state space is isomorphic to~$\Dt$.
For any $q$ and $B$ as above,
and for any countable abelian group~$G$,
we give a simple separable exact \ca~$D$,
not isomorphic to its opposite algebra,
with real rank zero and a unique tracial state,
such that $K_0 (D) \cong K_0 (B)$ and
$K_1 (D) \cong G \otimes_{\Z} {\mathbb{Z}} \big[ \tfrac{1}{q} \big]$.
We show that all the \ca{s} we construct
satisfy the Universal Coefficient Theorem.
We give further
information on the algebras described above, including
showing that the order on projections is determined by
traces, computing the Cuntz semigroups, and showing
that the algebras have stable rank one and tensorially
absorb the $q^{\infty}$ UHF algebra and the Jiang-Su
algebra.
The examples described above are not
the most general that can be obtained with our
method, but are chosen to illustrate the possibilities.
There are infinitely many primes $q$ such that $-1$
is not a square mod~$q$, so there are infinitely many
choices for $q$ covered by our examples.

Our results show an essential difference between
the class of nuclear  \ca{s} absorbing
the Jiang-Su algebra and the class of exact
\ca{s} absorbing the Jiang-Su algebra, and the importance
of nuclearity for the classication of simple \ca{s}.

Question~8.1 in~\cite{PhV} asked whether for any
UHF algebra~$B$ there exists a simple separable
exact \ca{} $D$ not isomorphic to its opposite
algebra that has the same K-theory as~$B$, and
the same properties as the algebra of~\cite{PhV}.
Our results provide a partial positive answer to this question.

The paper is organized as follows.
Section~\ref{Sec_Prelim} contains preliminaries.
In particular, we recall some relevant definitions
and constructions involving von Neumann algebras
and the Connes invariant.
In Section~\ref{asymphom} we recall the definition of the
continuous Rokhlin property for an action of a finite group
$G$ on a separable unital \ca.
The model
action of $G$ on the UHF algebra of type $\card (G)^{\infty}$
is an example of an action with this property.
Our main
result is that if $A$ is a separable, unital \ca{}
satisfying the Universal Coefficient Theorem, and $\alpha$
is an action of a finite abelian group $G$ on $A$ with
the continuous Rokhlin property, then the fixed point
algebra $A^{\af}$ and the crossed product $C^* (G,A,\alpha)$
satisfy the Universal Coefficient Theorem.
In Theorem 1.10
of \cite{Gd2}, Gardella has later extended this result to actions of
second countable compact groups.
In Section
\ref{Sec_Constr} we construct our basic example,
one algebra $D$ for each prime $q$ such that
 $- 1$ is not a square mod~$q$, satisfying
the same properties as the algebra in~\cite{PhV},
and whose $K$-theory is given by
$K_0 (D) \cong {\mathbb{Z}} \big[ \tfrac{1}{q} \big]$
and $K_1 (D) = 0$.
Moreover, we show that $D$ satisfies the Universal
Coefficient Theorem.
Section~\ref{Sec_Main} contains
the main step towards the proof that these algebras
are not isomorphic to their opposites.
Each of them has a unique tracial state.
We prove that the weak closure of $D$
in the Gelfand-Naimark-Segal representation
associated with this tracial state
is not isomorphic to its opposite algebra.
In Section~\ref{Sec_NotIsoOpp},
we tensor these basic examples with other
simple separable nuclear unital \ca{s}.
The main result of Section~\ref{Sec_Main}
also applies to such tensor products,
and we thus obtain the examples described above.
Section~\ref{Sec_Probs} contains some open problems.

The authors would like to thank E.~Gardella for
useful discussions about the Universal Coefficient Theorem
and for suggesting the argument used to show that the
algebra $D$ in Section \ref{Sec_Constr} satisfies the
Universal Coefficient Theorem.

Throughout,
we denote the circle $\{ \zt \in \C \colon | \zt | = 1 \}$ by~${\mathbb{T}}$.
We also denote the cyclic group $\Z / n \Z$ by $\Z_n$;
the $p$-adic integers will not be used in this paper.

\section{Preliminaries}\label{Sec_Prelim}

In this section,
we provide some background material about opposite algebras,
automorphisms of ${\mathrm{II}}_1$ factors,
the Connes invariant,
and the Cuntz semigroup.

First we recall the definition of the opposite algebra
and the conjugate algebra of a \ca~$A$.

\begin{definition}\label{D:2.6}
Let $A$ be a \ca.
The opposite algebra $A^{\mathrm{op}}$ is the \ca{}
which has the same vector space structure, norm, and adjoint
as~$A$, while the product of $x$ and $y$ in $A^{\mathrm{op}}$,
which we denote by $x \star y$ when necessary,
is given by $x \star y = y x$.
If $\om \colon A \to \C$
is a linear functional,
then we let $\om^{\mathrm{op}}$
denote the same map but regarded
as a linear functional
$\om^{\mathrm{op}} \colon A^{\mathrm{op}} \to \C$.
The conjugate algebra $A^{\mathrm{c}}$ is the \ca{}
whose underlying vector space structure is the conjugate of~$A$,
that is,
the product of $\lambda \in \C$ and $x \in A^{\mathrm{c}}$ is equal to
${\overline{\lambda}} x$ (as evaluated in~$A$),
and whose ring structure, adjoint, and norm are the same as for~$A$.
\end{definition}

\begin{remark}\label{R:2.7}
The map $x \mapsto x^{*}$ is an isomorphism from
$A^{\mathrm{c}}$ to $A^{\mathrm{op}}$.
\end{remark}

\begin{notation}\label{N_3704_GNS}
Let $A$ be a \ca,
and let $\om$ be a state on~$A$.
We denote the triple consisting of the Gelfand-Naimark-Segal
representation, its Hilbert space, and its standard cyclic vector
by $(\pi_{\om}, H_{\om}, \xi_{\om})$.

Also,
for any \ca{} or von Neumann algebra~$A$
and any tracial state $\ta$ on~$A$,
we denote the usual $L^2$-norm by
$\| x \|_{2, \ta} = ( \tau (x^{*} x) )^{1/2}$ for $x \in A$.
When no confusion can arise about the tracial state used,
we write $\| x \|_2$.
\end{notation}

It seems useful to make explicit the following fact,
which has been used implicitly in previous papers.

\begin{lemma}\label{L_3704_OppTr}
Let $A$ be a \ca,
and let $\ta$ be a tracial state on~$A$.
Then $\ta^{\mathrm{op}}$ is a tracial state on $A^{\mathrm{op}}$
and, as von Neumann algebras, we have
$\pi_{\ta^{\mathrm{op}}} (A^{\mathrm{op}})''
  \cong [\pi_{\ta} (A)'']^{\mathrm{op}}$.
\end{lemma}

\begin{proof}
The functional $\ta^{\mathrm{op}}$ is a state
because $A$ and $A^{\mathrm{op}}$
have the same norm and positive elements.
It is immediate that $\ta^{\mathrm{op}}$ is tracial.

Next, we claim that
$\| x \|_{2, \ta^{\mathrm{op}}} = \| x \|_{2, \ta}$
for all $x \in A$.
Indeed,
using the trace property at the third step,
\[
(\| x \|_{2, \ta^{\mathrm{op}}} )^2
  = \ta^{\mathrm{op}} (x^* \star x)
  = \ta (x x^*)
  = \ta (x^* x)
 = (\| x \|_{2, \ta} )^2.
\]

We can identify $\pi_{\ta} (A)''$ with the set of
elements in the Hausdorff completion of $A$ in $\| \cdot \|_{2, \ta}$
which are limits in $\| \cdot \|_{2, \ta}$ of norm bounded
sequences in~$A$,
and similarly with $\pi_{\ta^{\mathrm{op}}} (A^{\mathrm{op}})''$.
It follows from the claim that the identity map of~$A$
extends to a linear isomorphism
$\pi_{\ta} (A)'' \to \pi_{\ta^{\mathrm{op}}} (A^{\mathrm{op}})''$,
which is easily seen to preserve adjoints and reverse multiplication.
\end{proof}

To prove that our \ca{s}
are not isomorphic to their opposite algebras,
we will need some terminology and results
for the automorphisms of a ${\mathrm{II}}_1$~factor.

\begin{definition}\label{D:2.8}
For any von Neumann algebra~$M$,
we denote by ${\operatorname{Inn}} (M)$
the group of inner automorphisms of~$M$,
that is, the automorphisms of the form $\Ad (u)$
for some unitary $u \in M$.
Let $M$ be a ${\mathrm{II}}_1$~factor with separable predual.
Denote by $\tau$ the unique tracial state on $M$.
An automorphism $\ph$ of $M$ is {\emph{approximately inner}}
if there exists a sequence of unitaries $(u_n)_{n \in \N}$ in~$M$
such that $\Ad (u_n) \to \ph$ pointwise in $\| \cdot \|_{2}$.
Denote by ${\operatorname{\overline{Inn (M)}}}$
the group of approximately inner automorphisms of~$M$.
\end{definition}

Another important class of automorphisms
consists of the centrally trivial automorphisms of~$M$.

\begin{definition}\label{D:2.9}
Let $M$ be a ${\mathrm{II}}_1$ factor with
separable predual.
Let $\tau$ be the unique tracial state on $M$.
Recall that a bounded sequence $(x_n)_{n \in \N}$ in $M$ is central
if $\displaystyle{\lim_{n \to \infty} \| x_n a - a x_n \|_2 = 0}$
for all $a \in M$.
An automorphism $\ph$ of $M$ is said to be
{\emph{centrally trivial}}
if $\displaystyle{\lim_{n \to \infty} \| \ph (x_n) - x_n \|_2 = 0}$
for every central sequence $(x_n)_{n \in \N}$ in~$M$.
Let ${\operatorname{Ct}} (M)$ denote the set
of all centrally trivial automorphisms of~$M$.
\end{definition}

By the comments following Definition~3.1 in~\cite{Connes2},
the set ${\operatorname{Ct}} (M)$
is a normal subgroup of ${\operatorname{Aut}} (M)$.
It is obviously closed.

We recall below from~\cite{Connes1}
the definition of Connes invariant $\ch (M)$
of a ${\mathrm{II}}_1$ factor~$M$.
In~\cite{Connes1}, Connes uses centralizing
sequences to define the centrally trivial automorphisms.
For $\om \in M_*$ and
$x \in M$, we define $[ \om, x] \in M_*$
by $[ \om, x] (y) = \om (x y - y x)$
for $y \in M$.
A bounded sequence $(x_n)_{n \in \N}$
is then said to be centralizing
if $\displaystyle{\lim_{n \to \infty} \| [\om, x_n] \| = 0}$
for all $\om \in M_*$.
In general, centralizing sequences are the right
ones to use to define the Connes invariant.

A bounded sequence $(x_n)_{n \in \N}$ in a ${\mathrm{II}}_1$
factor~$M$ is central if and only if for some
(equivalently, for any) strong operator dense subset
$S \subseteq M$, we have $x_ny-yx_n  \to 0$ in the strong
operator topology for all $y \in S$.
In a ${\mathrm{II}}_1$~factor, we claim that
the central sequences are the same as the centralizing
sequences.
The implication from ($\beta$) to ($\gamma$)
in Proposition~2.8 of~\cite{Connes}
shows that centralizing sequences in~$M$ are central.
For the reverse,
in Proposition 2.8($\alpha$) of~\cite{Connes},
we take $\ph$ there to be the tracial state~$\ta$.
Since $[\ta, y] = 0$ for all $y \in M$,
the implication from ($\alpha$) to ($\gamma$)
there shows that central sequences in~$M$ are
centralizing.

\begin{definition}\label{D:2.10}
Let $M$ be a ${\mathrm{II}}_1$ factor with separable predual.
The {\emph{Connes invariant}} $\ch (M)$ is the subgroup
of the outer automorphism group
${\operatorname{Out}} (M) = \Aut (M) / {\operatorname{Inn}} (M)$
obtained as the center
of the image under the quotient map of
the group of approximately inner automorphisms.
\end{definition}

\begin{remark}\label{R:2.11}
If $M$ is isomorphic to its tensor product
with the hyperfinite ${\mathrm{II}}_1$~factor,
then $\ch (M)$ is the image in
$\Aut (M) / {\operatorname{Inn}} (M)$ of
${\operatorname{Ct}} (M) \cap {\operatorname{\overline{Inn (M)}}}$.
See~\cite{Connes1}.
\end{remark}

In general it is not easy to compute the Connes invariant
of a ${\mathrm{II}}_1$ factor.
For the hyperfinite ${\mathrm{II}}_1$ factor~$R$,
every centrally trivial automorphism is inner
by Theorem 3.2 (1) in~\cite{Connes3}, so $\ch (R) = \{ 0 \}$.
Moreover, any approximately inner automorphism
of the free group factor on $n$ generators
$\mathcal{L} (\mathbb{F}_n)$
is inner, so $\chi (\mathcal{L} (\mathbb{F}_n)) = \{ 0 \}$.
(See~\cite{VJones},  or Lemma 3.2 in~\cite{Viola}.)

A useful tool to compute the Connes invariant
of some ${\mathrm{II}}_1$ factors is the short exact
sequence introduced in~\cite{Connes1}.
Assume that $N$ is a ${\mathrm{II}}_1$ factor
without nontrivial hypercentral sequences, that is,
central sequences that asymptotically commute
in the $L^2$-norm with every central sequence of~$N$.
Let $G$ be a finite subgroup of ${\operatorname{Aut}} (N)$
such that $G \cap {\overline{\operatorname{Inn} (N)}} = \{ 1 \}$,
and let $\te \colon G \to {\operatorname{Aut}} (N)$ be the inclusion,
regarded as an action of $G$ on~$N$.
Define $K = G \cap {\operatorname{Ct}} (N)$ and let
$K^{\bot} \subseteq \widehat{G}$ be its annihilator, that is,
\[
K^{\bot} = \big\{f \colon G \to \mathbb{T} \colon
   {\mbox{$f$ is a homomorphism and $f |_K = 1$}} \big\}
 \subseteq \widehat{G}.
\]
For any $M$, let
$\xi_M \colon {\operatorname{Aut}} (M) \to {\operatorname{Out}} (M)$
denote the quotient map.
Let $H \subseteq {\operatorname{Aut}} (N)$ be the subgroup
\begin{equation}\label{Eq_2815_H}
H = \big\{ \Ad (u) \colon {\mbox{$u \in N$ is unitary and $\rh (u) = u$
        for all $\rh \in G$}} \big\}.
\end{equation}
Let $G \vee {\operatorname{Ct}} (N)$ be the subgroup of
${\operatorname{Aut}} (N)$
generated by $G \cup {\operatorname{Ct}} (N)$.
(It is closed since $G$ is finite and ${\operatorname{Ct}} (N)$
is closed and normal.)
Taking the closure in the topology of pointwise $L^2$-norm convergence,
let
\[
L = \xi_N \big( (G \vee {\operatorname{Ct}} (N))
   \cap \overline{H} \big)
  \subseteq {\operatorname{Out}} (N).
\]
Then the Connes short exact sequence
(Theorem~4 of~\cite{Connes1}) is
\begin{equation}\label{Connessequence}
\{ 1 \} \longrightarrow K^{\bot}
        \overset{\partial}{\longrightarrow} \chi (N \rtimes_{\te} G)
        \overset{\Pi}{\longrightarrow} L
        \longrightarrow \{ 1 \}.
\end{equation}

We briefly
describe the maps $\partial$ and $\Pi$ in this exact sequence.
We follow Section~5 in~\cite{Jones1}.
For $g \in G$, let $u_g \in N \rtimes_{\te} G$ be the standard unitary
in the crossed product associated to~$g$,
so that $\te_g = \Ad (u_g) |_N$.
Given an element $x$ in $N \rtimes_{\te} G$,
write it as $\displaystyle{\sum_{g \in G} a_g u_g}$
with $a_g \in N$ for $g \in G$.
For each $\varphi \colon G \to \mathbb{T}$ in $K^{\perp}$, define
$\Delta (\varphi) \colon N \rtimes_{\te} G \to N \rtimes_{\te} G$
by
\[
\Delta (\varphi) \left( \sum_{g \in G} a_g u_g \right)
 = \sum_{g \in G} \varphi (g) a_g u_g.
\]
Then
\[
\Delta (\varphi)
 \in \Ct (N \rtimes_{\te} G) \cap \overline{\Inn (N \rtimes_{\te} G)},
\]
and $\partial = \xi_{N \rtimes_{\te} G} \circ \Delta$.

To define~$\Pi$, for any $\sigma \in \chi (N \rtimes_{\te} G)$
choose an automorphism
$\alpha
 \in \Ct (N \rtimes_{\te} G) \cap \overline{\Inn (N \rtimes_{\te} G)}$
such that $\xi_{N \rtimes_{\te} G} (\alpha) = \sigma$.
Since $G \cap \Inn (N) = \{1 \}$, by Corollary~6 and Lemma~2 in
\cite{VJones} (or by Lemma 15.42 in \cite{EvKa}) there exist a
sequence $(u_n)_{n \in \N}$ of unitaries in $N^{\te}$
and a unitary $z \in N \rtimes_{\te} G$ such that
$\displaystyle{\alpha
 = \Ad (z) \circ \left[ \lim_{n \to \infty} \Ad (u_n) \right]}$.
Set $\varphi_{\sigma} = \Ad (z^*) \circ \alpha |_N$.
One can then show
that $\varphi_{\sigma} \in (G \vee \Ct (N)) \cap \overline{H}$, and
that the map $\xi_{N \rtimes_{\te} G} ( \varphi_{\sigma} )$ does not depend on
the choice of the representative $\alpha$ but only on the
class $\sigma$.
Therefore the map
$\Pi \colon \chi (N \rtimes_{\theta} G) \to L$, given
by $\Pi (\sigma) = \xi_{N \rtimes_{\te} G} (\varphi_{\sigma})$,
is well defined.

To see that $\Pi$ is surjective, let $\mu \in L$
and choose $\alpha \in (G \vee \Ct (N)) \cap \overline{H}$
such that $\xi_N (\alpha) = \mu$.
Since $\alpha$ is
the limit of automorphisms that commute with $G$, it also
commutes with $G$.
It follows that the map
\begin{equation}\label{betamu}
\et \left( \sum_{g \in G} a_g u_g \right)
  = \sum_{g \in G} \alpha (a_g) u_g
\end{equation}
is an automorphism of $N \rtimes_{\te} G$.
Moreover,
$\et
 \in \Ct (N \rtimes_{\te} G) \cap \overline{\Inn (N \rtimes_{\te} G)}$
and $\Pi (\xi_{N \rtimes_{\te} G} (\et)) = \mu$.

An important concept in the classification of
automorphisms of the hyperfinite ${\mathrm{II}}_1$ factor is
the one of obstruction to lifting, defined by Connes in
Section 1 of~\cite{Connes3}.
It will play a key role in showing that
our algebras are not isomorphic to their opposites.

\begin{definition}\label{D:2.12}
Let $M$ be a ${\mathrm{II}}_1$ factor
and let $\alpha$ be an automorphism of~$M$.
Let $n$ be the smallest nonnegative integer
such that there is a unitary $u \in M$
with $\alpha^n = \Ad (u)$.
If no power of $\alpha$ is inner, we set $n = 0$.
Since $M$ is a factor, it is easy to check that
there is $\ld \in \C$
such that $\ld^n = 1$ and
$\alpha (u) = \lambda u$.
(Simply apply
$\alpha^{n + 1} = \alpha \circ \alpha^n = \alpha^n \circ \alpha$
to any element $x$ in $M$.)
We call $\lambda$ the {\emph{obstruction to lifting}} of
$\alpha$, and refer to the pair $(n, \lambda)$
as the outer invariant of $\alpha$.
\end{definition}

Next, we recall what it means
for the order on projections to be determined by traces.
Let $A$ be a \ca{} and denote by
$M_n (A)$ the $n \times n$ matrices with entries in~$A$.
Let $M_{\infty} (A)$ denote the algebraic direct limit
of the sequence $(M_n (A), \varphi_n)_{n \in \N}$,
in which
$\varphi_n \colon M_n (A) \to M_{n + 1} (A)$
is defined by
$a \mapsto
 \left( \begin{smallmatrix} a & 0 \\
   0 & 0 \end{smallmatrix} \right)$.
Denote by $\T (A)$ the set of tracial states of $A$.

\begin{definition}\label{D:OrdDetTr}
We say that
{\emph{the order on projections over $A$ is determined by traces}}
if whenever $p_1, p_2 \in M_{\infty} (A)$ are projections
such that $\tau (p_1) < \tau (p_2)$ for every $\tau$ in $\T (A)$,
then $p_1$ is Murray-von Neumann subequivalent to~$p_2$.
\end{definition}

We conclude this section by recalling
the definitions of Cuntz subequivalence
and the Cuntz semigroup.
See Section~2 of~\cite{BPT}
and the references there for the definitions
below and the proofs of the assertions made here.

\begin{definition}\label{D:2.13}
Let $A$ be a \ca{} and let $a, b \in M_{\infty} (A)_{+}$.
We say that $a$ is Cuntz subequivalent to~$b$,
denoted $a \precsim b$,
if there exists a sequence $(v_n)_{n \in \N}$ in $M_{\infty} (A)$
such that
$\displaystyle{\lim_{n \to \infty} \| v_n b v_n^{*} - a \| = 0}$.
If $a \precsim b$ and $b \precsim a$ we say that $a$
is Cuntz equivalent to $b$ and write $a \sim b$.
\end{definition}

Cuntz equivalence is an equivalence relation,
and we write $\langle a \rangle$ for the equivalence class of~$a$.

\begin{definition}\label{D:2.14}
Let $A$ be a \ca.
The {\emph{Cuntz semigroup}} of~$A$
is $W (A) = M_{\infty} (A)_{+} / {\sim}$.
We define a semigroup operation on $W (A)$ by
\[
\langle a \rangle + \langle b \rangle
 = \left\langle \left(  \begin{matrix}
  a     &  0        \\
  0     &  b
\end{matrix} \right) \right\rangle
\]
and a partial order by $\langle a \rangle \leq \langle b \rangle$
if and only if $a \precsim b$.
With this structure $W (A)$ becomes a positively ordered abelian
semigroup with identity.
\end{definition}

Usually,
it is hard to compute the Cuntz semigroup of a \ca,
but the following remark computes
$W (A)$ for the \ca{s}~$A$ of interest here.

\begin{remark}\label{R:2.15}
Denote the Jiang-Su algebra by $Z$.
Assume that $A$ is a simple unital exact stably finite \ca{}
which is $Z$-stable, that is, $Z \otimes A \cong A$.
Let $V (A)$ be the Murray-von Neumann semigroup of~$A$.
For any compact convex set~$\Dt$,
let ${\operatorname{LAff}}_{\operatorname{b}} (\Dt)_{++}$
denote the set of bounded strictly positive
lower semicontinuous affine functions on~$\Delta$.
By Corollary 5.7 of~\cite{BPT}, we have
\[
W (A) \cong V (A)
 \amalg {\operatorname{LAff}}_{\operatorname{b}} (\T (A))_{++}.
\]
The addition and order on the disjoint union
are defined as follows.
On each part of the disjoint union,
the addition and order are as usual.
For the other cases,
for $x \in V (A)$
define ${\widehat{x}} \colon \T (A) \to [0, \infty)$
by $\widehat{x} (\tau) = \tau (x)$ for $\tau \in \T (A)$.
Now let $x \in V (A)$ and
$y \in {\operatorname{LAff}}_{\operatorname{b}} (\T (A))_{++}$.
Then $x + y$ is the function
${\widehat{x}} + y
 \in {\operatorname{LAff}}_{\operatorname{b}} (\T (A))_{++}$.
Also, $x \leq y$ if and only if
$\widehat{x} (\tau) < y (\tau)$
for all $\tau \in \T (A)$,
and $y \leq x$ if and only if $y (\tau) \leq \widehat{x} (\tau)$
for all $\tau \in \T (A)$.
\end{remark}

\section{The continuous Rokhlin property and the
  Universal Coefficient Theorem}\label{asymphom}

We recall the definition of asymptotic homomorphism, and what
it means for an action of a finite group on a \ca{} to
have the continuous Rokhlin property.
The main result of this
section is that if $A$ is a separable unital \ca{}
satisfying the Universal Coefficient Theorem, and $\af$ is an
action of a finite group $G$ on $A$ satisfying the continuous
Rokhlin property, then the fixed point algebra $A^{\af}$ and
the crossed product $C^* (G, A, \af)$ satisfy the Universal
Coefficient Theorem.
Gardella has later extended this result to
action of second countable compact groups.
(See Theorem 1.10 in \cite{Gd2}.)

\begin{definition}\label{ContinuousRokhlin}
Let $A$ be a separable, unital \ca, and let
$\alpha \colon G \to \Aut (A)$ denote an action of a finite group
$G$ on $A$.
We say that $\alpha$ has the
{\emph{continuous Rokhlin property}} if
there exist continuous functions $t \to e_g^{(t)}$
from $[0, \infty)$ to $A$,
for $g \in G$,
such that:
\begin{enumerate}
\item\label{orthogonality}
For each $t \in [0, \infty)$, $\big( e_g^{(t)} \big)_{g \in G}$
is a family of mutually orthogonal projections with
${\displaystyle{\sum_{g \in G} e_g^{(t)} = 1}}$.
\item\label{shift}
${\displaystyle{\lim_{t \to \infty}
   \big\| \alpha_g \big( e_h^{(t)} \big) - e_{g h}^{(t)} \big\| = 0}}$
for every $g, h \in G$.
\item\label{commutativity}
For any $g \in G$ and $a \in A$, we have
$\displaystyle{\lim_{t \to \infty}
 \big\| e_g^{(t)} a - a e_{g}^{(t)} \big\| = 0}$.
\end{enumerate}
\end{definition}

\begin{lem}\label{L_5Y02_CtRkTP}
Let $A$ and $B$ be separable unital \ca{s},
let $G$ be a finite group,
and let $\af \colon G \to \Aut (A)$
and $\bt \colon G \to \Aut (B)$
be actions of $G$ on $A$ and~$B$.
Assume that $\af$ has the continuous Rokhlin property.
Let $A \otimes B$ be any C$^*$~tensor product
on which the tensor product action $g \mapsto \af_g \otimes \bt_g$
is defined.
Then $\af \otimes \bt$ has the continuous Rokhlin property.
\end{lem}

\begin{proof}
Let $\big( e_g^{(t)} \big)_{g \in G, \, t \in [0, \infty)}$
be a family of projections in~$A$
as in Definition~\ref{ContinuousRokhlin}
for the action $\af$.
Then $\big( e_g^{(t)} \otimes 1 \big)_{g \in G, \, t \in [0, \infty)}$
is a family of projections in~$A \otimes B$
as in Definition~\ref{ContinuousRokhlin}
for the action $\af \otimes \bt$.
\end{proof}

The following example of an action of a finite group with the
continuous Rokhlin property will be needed in
Proposition~\ref{P_2_6}.

\begin{example}\label{modelaction}
Let $G$ be a topological group, and let $d_1, d_2, \ldots \in \N$.
Let $\rh = \big( \rh^{(1)}, \rh^{(2)}, \ldots \big)$ be a sequence
of unitary representations \mbox{$\rh^{(k)} \colon G \to L (\C^{d_k})$}
of~$G$.
Define actions $\nu^{(k)} \colon G \to \Aut (L (\C^{d_k}))$
by $\nu^{(k)}_g (a) = \rh^{(k)} (g) a \rh^{(k)} (g)^*$ for
$k \in \N$, $g \in G$, and $a \in L (\C^{d_k})$.
Let $B_{\rh}$
be the UHF~algebra
\[
B_{\rh} = \bigotimes_{k = 1}^{\infty} L (\C^{d_k}),
\]
and let $\mu^{\rh} \colon G \to \Aut (B_{\rh})$
be the product type action given by
\[
\mu^{\rh}_g = \bigotimes_{k = 1}^{\infty} \nu^{(k)}_g.
\]

For a fixed unitary representation $\rh \colon G \to L (\C^d)$,
we abbreviate $(\rh, \rh, \ldots)$ to~$\rh$,
so that $B_{\rh}$ is the $d^{\I}$~UHF algebra,
and the action is given by
\[
g \mapsto \mu^{\rh}_g
 = \bigotimes_{k = 1}^{\infty} \Ad (\rh (g)) \in \Aut (B_{\rh}).
\]
When $G$ is finite with $\card (G) = d$,
and $\rh$ is the regular representation
$\ld \colon G \to L (l^2 (G))$, we write
\[
B_G = \bigotimes_{k = 1}^{\infty} L (l^2 (G))
\andeqn
g \mapsto \mu_g^{G}
  = \bigotimes_{k = 1}^{\infty} \Ad (\ld (g)) \in \Aut (B_{G}),
\]
and when $\rh$ is the direct sum $\ld^m$ of $m$ copies of~$\ld$,
we write
\[
B_{G, m} = \bigotimes_{k = 1}^{\infty} L (l^2 (G)^m)
\andeqn
g \mapsto \mu_g^{G, m} = \bigotimes_{k = 1}^{\infty} \Ad (\ld^m (g))
 \in \Aut (B_{G, m}).
\]
These are product type actions of $G$
on the $d^{\I}$~and $(m d)^{\I}$~UHF algebras.
\end{example}

We fix some notation.
For any index set~$S$ and $s \in S$, we
denote by $\dt_s \in l^2 (S)$ the standard basis vector,
determined by
\[
\dt_s (t) = \begin{cases}
   1 & \hspace*{1em} t = s
        \\
   0 & \hspace*{1em} t \neq s.
\end{cases}
\]

\begin{lem}\label{L:PTypeCRP}
Let $G$ be a finite group.
Then the action
$\mu^{G, m} \colon G \to \Aut (B_{G, m})$ of Example
\ref{modelaction} has the continuous Rokhlin property.
\end{lem}

\begin{proof}
We use the notation above.
Recall that $\ld^m$ is the direct sum of $m$ copies of
the regular representation of~$G$,
and define $\nu \colon G \to \Aut ( L (l^2 (G)^m) )$
by $\nu_g = \Ad (\ld^m (g))$.

We begin with a construction involving just two tensor factors.
Let
\[
v \colon l^2 (G)^m \otimes l^2 (G)^m \to l^2 (G)^m \otimes l^2 (G)^m
\]
be the unitary determined by $v (\xi \otimes \et) = \et \otimes \xi$
for $\xi, \et \in l^2 (G)^m$.
Equip
\[
L \big( l^2 (G)^m \otimes l^2 (G)^m \big)
   = L (l^2 (G)^m) \otimes L (l^2 (G)^m)
\]
with the action $g \mapsto \nu_g \otimes \nu_g$.
Then $v$ is $G$-invariant.
Since \mbox{$L \big( l^2 (G)^m \otimes l^2 (G)^m \big)^G$}
is finite dimensional,
there is a \ct\  path $t \mapsto z_t$ of $G$-invariant unitaries
in $L \big( l^2 (G)^m \otimes l^2 (G)^m \big)$
such that $z_0 = 1$ and $z_1 = v$.

For every $g \in G$,
let $\delta_g$ be the corresponding standard basis
vector in $l^2 (G)$,  and let $p_g \in L (l^2 (G)^m)$ be the
projection on the $m$~dimensional subspace spanned by
the standard basis vectors
$\dt_{g, j} = (0, \ldots, 0, \delta_g, 0, \ldots, 0)$
for $1 \leq j \leq m$, where $\delta_g$ is in the $j$-th position.
Then $\nu_g (p_h) = p_{g h}$ for every $g, h \in G$,
and $\displaystyle{\sum_{g \in G} p_g = 1.}$
For $n \in \Nz$, $t \in [n, \, n + 1]$, and $g \in G$,
we define
\[
e_g^{(t)} = 1 \otimes 1 \otimes \cdots \otimes 1 \otimes
      z_{t - n} (p_g \otimes 1) z_{t - n}^* \otimes 1 \otimes
      1 \otimes 1 \otimes \cdots
    \in B_{G, m},
\]
with the expression $z_{t - n} (p_g \otimes 1) z_{t - n}^*$
occupying the two positions $n + 1$ and $n + 2$ in the
tensor product.
It is clear that $e_g^{(t)}$ is a \pj\  and that
$\displaystyle{\sum_{g \in G} e_g^{(t)} = 1}$ for all $t \in [0, \I)$.
Since $(\nu_g \otimes \nu_g) (z_{t - n}) = z_{t - n}$
and $\nu_g (p_h) = p_{g h}$,
one easily checks that
$\mu^{G, m}_g \big( e_h^{(t)} \big) = e_{g h}^{(t)}$
for all $g, h \in G$ and $t \in [0, \I)$.
Finally, if
\[
a \in \bigotimes_{k = 1}^{N} L (l^2 (G)^m)
   \S \bigotimes_{k = 1}^{\infty} L (l^2 (G)^m)
\andeqn
t \geq N,
\]
then $e_g^{(t)}$ exactly commutes with~$a$.

Now take $b \in B_{G, m}$.
For every $\varepsilon > 0$ there
exist $N \geq 1$ and
$\displaystyle{a \in \bigotimes_{k = 1}^{N} L (l^2 (G)^m)}$
such that $\| b - a \| < \frac{\varepsilon}{2}$.
Suppose $t \geq N$.
Then $e_g^{(t)} a = a e_g^{(t)}$ by the previous paragraph,
so
\[
\big\| e_g^{(t)} b - b e_g^{(t)} \big\|
 \leq 2 \| b - a \|
     + \big\| e_g^{(t)} a - a e_g^{(t)} \big\|
 = 2 \| b - a \|
 < \varepsilon.
\]
This completes the proof.
\end{proof}

The following result will not be used,
but it is easy to derive from known results
and provides motivation for the idea that the action
we construct in Section~\ref{Sec_Constr}
should have the continuous Rokhlin property.
The result we actually need
is in Proposition~\ref{P_2_6} below.
It is known that there are actions on simple \ca{s}
which have the Rokhlin property
but not the continuous Rokhlin property.
Giving an example here would take us too far afield.

\begin{proposition}\label{P_2_6X}
Let $G$ be a finite group.
Let $A$ be a simple separable unital nuclear
\ca{} satisfying the Universal Coefficient Theorem
which, in addition,
is either purely infinite or tracially AF
in the sense of~\cite{Ln1}.
Let $\af \colon G \to \Aut (A)$
be an action with the Rokhlin property.
Assume that for all $g \in G$,
the maps $(\af_g)_* \in \Aut (K_* (A))$
are the identity maps.
Then $\af$ has the continuous Rokhlin property.
\end{proposition}

\begin{proof}
Apply Theorem~3.4 of~\cite{Iz2} in the purely infinite case,
and Theorem~3.5 of~\cite{Iz2}
in the tracially AF (tracial rank zero) case,
to show that $\af$ is conjugate to its tensor product
with the action $\mu^G$ of Example~\ref{modelaction}.
Since $\mu^G$ has the continuous Rokhlin property
(by Lemma~\ref{L:PTypeCRP}),
it follows from Lemma~\ref{L_5Y02_CtRkTP}
that $\af$ has the continuous Rokhlin property.
\end{proof}

We now recall the definition of an asymptotic homomorphism.

\begin{definition}\label{AsympMorophism}
Let $A$ and $B$ be \ca{s}.
An asymptotic
homomorphism from $A$ to $B$ is a family of maps
$\psi_t \colon A \to B$, indexed by $t \in [0, \infty)$,
satisfying the following conditions:
\begin{enumerate}
\item
For all $a \in A$ the map $t \mapsto \psi_t (a)$,
from $[0, \infty)$ to $B$,
is continuous.
\item
For all $a, \, b \in A$ and $\lambda \in {\mathbb{C}}$ one has
\[
\lim_{t \to \infty}
  \| \psi_t (a + b) - \psi_t (a) - \psi_t (b) \| = 0,
\]
\[
\lim_{t \to \infty} \| \psi_t (\lambda a) - \lambda \psi_t (a) \| = 0,
\]
\[
\lim_{t \to \infty} \| \psi_t (ab) - \psi_t (a) \psi_t (b) \| = 0,
\]
and
\[
\lim_{t \to \infty} \| \psi_t (a^*) - \psi_t (a)^* \| = 0.
\]
\end{enumerate}
\end{definition}

Next we show that, given a separable unital \ca{}
and an action of a finite group $G$ on $A$ with the
continuous Rokhlin property, there exists a unital
completely positive asymptotic homomorphism
$t \mapsto \psi_t$ from
$A$ to $A^{\af}$
which is a left inverse for the inclusion.
The following argument was
suggested by E.~Gardella.
It replaces an earlier argument in which
$\ps_t$ was not completely positive.

\begin{proposition}\label{complposasymmorp}
Let $A$ be a separable unital \ca.
Let $G$ be a
finite group and let $\af \colon G \to \Aut (A)$ be an action with the
continuous Rokhlin property.
Denote by $A^{\af}$ the fixed point
algebra, and let $\iota \colon A^{\af} \to A$ be the canonical
inclusion.
Then there exists a unital completely positive
asymptotic homomorphism $t \mapsto \psi_t \colon A \to A^{\af}$ for
$t \in [0, \infty)$ such that
\[
\lim_{t \to \infty} \| (\psi_t \circ \io) (a) - a \| = 0
\]
for all $a \in A^{\af}$.
\end{proposition}

\begin{proof}
Given \ca{s} $A$ and $B$ and a map
$\psi \colon A \to B$,
we denote by $\psi^{(n)}$ the map
from $M_n (A)$ to $M_n (B)$ defined by
$\psi^{(n)} (a) = (\psi (a_{j, k}))_{j, k = 1}^n$ for
$a = (a_{j, k})_{j, k = 1}^n \in M_n (A)$.

Let $\big( e_g^{(t)} \big)_{g \in G, \, t \in [0, \infty)}$
be a family
of projections as in Definition~\ref{ContinuousRokhlin}.
For
$t \in [0, \infty)$ define a map $\rh_t \colon A \to A$ by
\[
\rh_t (a) = \sum_{g \in G} e_g^{(t)} \af_g (a) e_g^{(t)}
\]
for $a \in A$.
We claim that $t \mapsto \rh_t$
is a unital completely positive asymptotic homomorphism
from $A$ to~$A$
such that
\[
\lim_{t \to \infty} \| (\rh_t \circ \io) (a) - a \| = 0
\]
for all $a \in A^{\af}$.

Obviously $\rh_t$ is unital.
Moreover, $\rh_t$ is completely positive since if
$a = (a_{j, k})_{j, k = 1}^n \in M_n (A)_{+}$ and
\[
p_g^{(t)}
 = \operatorname{diag}
    \big( e_g^{(t)}, e_g^{(t)}, \ldots, e_g^{(t)} \big )\in M_n(A)
\]
denotes the diagonal matrix with the element $e_g^{(t)}$
everywhere on the diagonal, then
\[
\rh_t^{(n)} (a)
  = \sum_{g \in G} p_g^{(t)} \big( \af_g (a_{j, k}) \big)_{j, k = 1}^n
     p_g^{(t)}
 \geq 0
\]
for every $t \in [0, \infty)$.
To show that $\rh = (\rh_t)_{t \in [0, \infty)}$
is an asymptotic homomorphism, observe that
$t \mapsto \rh_t (a)$ is clearly continuous for every
$a \in A$,
and that for every $a, b \in A$ and $\lambda \in {\mathbb{C}}$
we have $\rh_t (\lambda a + b) = \lambda \rh_t (a) + \rh_t (b)$
and $\rh_t (a^*) = \rh_t (a)^*$.
Moreover,
\begin{equation}\label{Eq_5Y06_AppMult}
\begin{split}
& \lim_{t \to \infty} \| \rh_t (a b) - \rh_t (a) \rh_t (b) \| \\
& \hspace{3em} {\mbox{}}
  = \lim_{t \to \infty} \Biggl\| \sum_{g \in G}
          e_g^{(t)} \af_g (a b) e_g^{(t)}
        - \sum_{g \in G} e_g^{(t)} \af_g (a) e_g^{(t)}
                 \af_g (b) e_g^{(t)} \Biggr\|
\\
& \hspace{3em} {\mbox{}}
  \leq \lim_{t \to \infty} \sum_{g \in G} \| \af_g (a) \|
      \big\| \af_g (b) e_g^{(t)} - e_g^{(t)} \af_g (b) \big\|
  = 0.
\end{split}
\end{equation}
Lastly, for every $a \in A^{\af}$ we have
\begin{equation}\label{Eq_5Y06_LeftInv}
\begin{split}
\lim_{t \to \infty} \| \rh_t (a) - a \|
& = \lim_{t \to \infty} \Biggl\| \sum_{g \in G} e_g^{(t)} a e_g^{(t)}
      - \sum_{g \in G} e_g^{(t)} a \Biggr\|
\\
& \leq \lim_{t \to \infty}
       \sum_{g \in G} \big\| a e_g^{(t)} - e_g^{(t)} a \big\|
  = 0.
\end{split}
\end{equation}
The claim is proved.

It follows from the definition of $\rh_t$
and the relation~(\ref{shift}) in Definition~\ref{ContinuousRokhlin}
that
\begin{equation}\label{Eq_5Y06_AppInv}
\lim_{t \to \infty} \| \af_g (\rh_t (a)) - \rh_t (a) \| = 0
\end{equation}
for all $a \in A$ and $g \in G$.

Now let $E \colon A \to A^{\af}$ be the conditional expectation
given by
\[
E (a) = \frac{1}{\card (G)} \sum_{g \in G} \af_g (a)
\]
for $a \in A$.
For
$t \in [0, \infty)$ define a map $\ps_t \colon A \to A^{\af}$ by
$\ps_t (a) = E (\rh_t (a))$ for $a \in A$.
Obviously $\ps_t (a) \in A^{\af}$ for all $a \in A$,
and $\ps_t$ is unital, linear, and completely positive.
Also, $t \mapsto \ps_t (a)$ is clearly continuous for every
$a \in A$.
Since $G$ is finite,
it follows from~(\ref{Eq_5Y06_AppInv})
that
${\displaystyle{\lim_{t \to \infty} \| \rh_t (a) - \ps_t (a) \| = 0}}$
for all $a \in A$.
Combining this relation with~(\ref{Eq_5Y06_AppMult})
gives
${\displaystyle{\lim_{t \to \infty}
   \| \ps_t (a b) - \ps_t (a) \ps_t (b) \| = 0}}$
for all $a, b \in A$.
Combining it with~(\ref{Eq_5Y06_LeftInv})
gives
${\displaystyle{\lim_{t \to \infty} \| \ps_t (a) - a \| = 0}}$
for all $a \in A^{\af}$.
This completes the proof.
\end{proof}

\begin{proposition}\label{UCTcrossedproduct}
Let $A$ be a separable unital \ca,
and let $G$ be a finite group.
Let
$\af \colon G \to \Aut (A)$ be an action with
the continuous Rokhlin property.
Assume that $A$ satisfies the Universal Coefficient Theorem.
Then $A^{\af}$ and $A \rtimes_{\af} G$
satisfy the Universal Coefficient Theorem.
\end{proposition}

If $A$ is simple and nuclear,
then one does not need the continuous Rokhlin property;
the Rokhlin property suffices.
See Corollary~3.9 of~\cite{OsaPhil} for the crossed product
and, for actions of second countable compact groups,
see Theorem~3.13 of~\cite{Gd}.
Moreover,
Gardella, Hirshberg, and Santiago
proved (Theorem 4.17 in \cite{GHS}) that if $A$ is nuclear
and satisfies the Universal Coefficient Theorem,
$G$ is a second countable compact group
with finite covering dimension,
and $\alpha \colon G \to \Aut (A)$ has finite Rokhlin dimension with
commuting towers, then $A^{\alpha}$ and $A \rtimes_{\alpha} G$
also satisfy the Universal Coefficient Theorem.

\begin{proof}[Proof of Proposition~\ref{UCTcrossedproduct}]
Denote the suspension of a \ca{} $A$ by $SA$.
Let $K$ be the algebra of compact operators.
By Theorem~4.2 in~\cite{HLT}, for every pair of separable
\ca{s} $A$ and $B$, the group $KK (A, B)$ is
canonically isomorphic to the group of homotopy classes of
completely positive asymptotic homomorphisms from
$K \otimes S A$ to $K \otimes S B$.
Let $\io \colon A^{\af} \to A$ be the inclusion.
Proposition~\ref{complposasymmorp} implies that the
group homomorphism $\psi^* \colon KK (A^{\af}, B) \to KK (A, B)$
induced by the unital completely positive asymptotic
homomorphism $(\psi_t)_{t \in [0, \infty)}$ obtained there
satisfies
$\iota^* \circ \psi^* = \operatorname{id}_{KK (A^{\af}, B)}$.
In particular, $\psi^*$ is naturally split injective with left
inverse $\iota^* \colon KK (A, B) \to KK (A^{\af}, B)$.

By hypothesis, $A$ satisfies the Universal Coefficient Theorem
(Theorem~1.17 of~\cite{RS}).
That is, let $B$ be any separable \ca.
Let
\[
\gm_{A, B} \colon KK (A, B) \to \operatorname{Hom} (K_* (A), \, K_* (B))
\]
and
\[
\kp_{A, B} \colon \operatorname{Ker} (\gamma_{A, B}) \to
\operatorname{Ext} (K_* (A), \, K_{*+1} (B) )
\]
be as described before Theorem~1.17 of~\cite{RS}
(and called $\gm (A, B)$ and $\kp (A, B)$ in~\cite{RS}
when $A$ and $B$ must be specified).
Then
$\dt_{A, B} = \kp_{A, B}^{-1}$ exists, and
there is a (natural) short exact sequence
\[
\begin{split}
0 \longrightarrow \operatorname{Ext} (K_* (A), \, K_{*+1} (B))
& \stackrel{\delta_{A, B}}{\longrightarrow} KK_* (A, B)
\\
& \stackrel{\gamma_{A, B}}{\longrightarrow}
    \operatorname{Hom} (K_* (A), \, K_* (B))
\longrightarrow 0.
\end{split}
\]
(Naturality is Theorem~4.4 of~\cite{RS}.)

Consider the commutative diagram
\[
\begin{tikzcd}
KK (A, B) \arrow{d}{\iota^*}
 \arrow{r}{\gamma_{A, B}} & \arrow{d}{\iota^*}
       \operatorname{Hom} (K_* (A), \, K_* (B))
\\
KK (A^{\af}, B) \arrow{r}{\gamma_{A^{\af}, B}} &
\operatorname{Hom} (K_* (A^{\af}), \, K_* (B)),
\end{tikzcd}
\]
in which the vertical maps are induced by
$\io \colon A^{\af} \to A$ and the horizontal ones
are from the Universal Coefficient Theorem.
The map $\gamma_{A, B}$ is surjective because $A$ satisfies
the Universal Coefficient Theorem,
and the right vertical map is surjective because it is a left inverse
of $\ps^*$,
so $\gamma_{A^{\af}, B}$ is surjective.

Now consider the following commutative diagram,
in which the horizontal maps are from the Universal Coefficient Theorem:
\[
\begin{tikzcd}
\operatorname{Ker} (\gamma_{A, B})
\arrow{r}{\kp_{A, B}}
  & \operatorname{Ext} (K_* (A), \, K_{*+1} (B) )  \\
\operatorname{Ker} (\gamma_{A^{\af}, B}) \arrow{u}{\psi^*}
\arrow{r}{\kp_{A^{\af}, B}} &  \arrow{u}{\psi^*}
\operatorname{Ext} (K_* (A^{\af}), \, K_{*+1} (B) ).
\end{tikzcd}
\]
The map $\kp_{A, B}$ is injective because $A$ satisfies
the Universal Coefficient Theorem,
and the left vertical map is injective
since it has a left inverse~$\io^*$,
so $\kp_{A^{\af}, B}$ is injective.

Lastly, the argument used to prove that $\gamma_{A^{\af}, B}$
is surjective, applied to the commutative diagram
\[
\begin{tikzcd}
\operatorname{Ker} (\gamma_{A, B}) \arrow{d}{\iota^*}
\arrow{r}{\kp_{A, B}} & \operatorname{Ext} (K_* (A), \, K_{*+1} (B) )
\arrow{d}{\iota^*} \\
\operatorname{Ker} (\gamma_{A^{\af}, B})
\arrow{r}{\kp_{A^{\af}, B}} &
\operatorname{Ext} (K_* (A^{\af}), \, K_{*+1} (B) ),
\end{tikzcd}
\]
shows that $\kp_{A^{\af}, B}$ is
surjective.
Therefore $A^{\af}$ satisfies the Universal
Coefficient Theorem.

Now we consider the crossed product $A \rtimes_{\af} G$.
By the Proposition in~\cite{Rs},
$A^{\af}$ is isomorphic to a corner of $A \rtimes_{\af} G$.
When $\af$ has the Rokhlin property,
Corollary 2.15 of~\cite{Gd}
implies that
this corner is strongly Morita equivalent to $A \rtimes_{\af} G$.
(This is saturation of the action,
a weaker condition than hereditary saturation
as proved in~\cite{Gd}.)
Since strong Morita equivalence
preserves the class of algebras satisfying
the Universal Coefficient Theorem,
we conclude that $A \rtimes_{\af} G$
satisfies the Universal Coefficient Theorem.
\end{proof}

The following argument was suggested by E.~Gardella.

\begin{proposition}\label{P_2_6}
Let $G$  be a finite group.
Let $A$ be a unital
separable \ca{} which absorbs the UHF
algebra of type $\card (G)^{\infty}$.
Suppose the action
$\alpha \colon G \to \operatorname{Aut} (A)$
has the Rokhlin property.
Then:
\begin{enumerate}
\item\label{P_2_6_Stab}
Taking $\mu^G$ to be the product type action of
Example \ref{modelaction}, we have an equivariant isomorphism
\[
(G, A, \alpha)
 \cong \big( G, \, M_{\card (G)^{\infty}} \otimes A, \,
            \mu^G \otimes \alpha \big).
\]
\item\label{P_2_6_CRP}
The action $\alpha$ has the continuous Rokhlin property.
\end{enumerate}
\end{proposition}

\begin{proof}
We prove~(\ref{P_2_6_Stab}).
We will need to cite theorems which use central sequence algebras,
so we state notation for them.
For a separable unital \ca~$A$,
we define $A^{\I} = C_{\mathrm{b}} (\N, A) / C_0 (\N, A)$,
and we regard $A$ as a subalgebra of $A^{\I}$ via
its embedding in $C_{\mathrm{b}} (\N, A)$
as the algebra of constant sequences.
Then $A' \cap A^{\I}$ is the relative commutant of this image of~$A$.
(It is written $A_{\I}$ in~\cite{MI}.
See Section~2.1 there.)
For $\om \in \bt \N \setminus \N$,
if in place of $C_0 (\N, A)$ we use
\[
\left\{ a = (a_n)_{n \in \N} \in C_{\mathrm{b}} (\N, A) \colon
   \lim_{n \to \om} a_n = 0 \right\},
\]
we call the quotient~$A^{\om}$.
The image of the constant sequences here
can also be identified with~$A$,
and we again get a relative commutant $A' \cap A^{\om}$.
There is an obvious surjective map $A^{\I} \to A^{\om}$,
which gives a unital \hm{} $A' \cap A^{\I} \to A' \cap A^{\om}$.

We now fix any $\om \in \bt \N \setminus \N$.

Since $A$ absorbs $M_{\card (G)^{\infty}}$,
Proposition~2.8  in~\cite{MI}
and the comment at the beginning of the proof of
Proposition 2.9 in~\cite{MI}
provide an injective unital \hm{}
$M_{\card (G)^{\infty}} \to A^{\infty} \cap A'$.
Since $M_{\card (G)^{\infty}}$ is simple,
it follows from the previous paragraph
that there is an injective unital \hm{}
$M_{\card (G)^{\infty}} \to A^{\om} \cap A'$.
Lemma 3.12  in~\cite{KirPhi},
the proof of
Proposition 3.13 in~\cite{KirPhi},
and the fact that in Lemma 0.5 in~\cite{KirPhi}
the isomorphism is approximately unitarily equivalent to
the given \hm{}
(see the proof of Proposition~A in~\cite{Rr3}),
now provide a unital isomorphism
$\varphi \colon M_{\card (G)^{\infty}} \otimes A \to A$
and unitaries $w_n$ in $A$ for $n \in \N$ such that
\begin{equation}\label{Eq_5Y06_Lim0}
\lim_{n \to \infty} \| w_n \varphi (1 \otimes a) w_n^* - a \| = 0
\end{equation}
for every $a \in A$.

Define an action
$\beta \colon G \to \operatorname{Aut} (A)$ by
$\beta_g
 = \varphi \circ ( \mu_g^G \otimes \alpha_g ) \circ \varphi^{-1}$
for $g \in G$.
We claim that $\beta_g$ is approximately unitarily
equivalent to $\alpha_g$ for every $g \in G$.
Set $v_n = w_n \beta_g (w_n^*)$ for $n \in \N$.
Let $g \in G$ and $a \in A$.
For $n \in \N$ we have
\begin{align*}
\| v_n \beta_g (a) v_n^* - \alpha_g (a) \|
&  = \| w_n \beta_g (w_n^* a w_n) w_n^* - \alpha_g (a) \|
\\
& \leq \| w_n \beta_g (w_n^* a w_n) w_n^*
         - w_n \beta_g (\varphi (1 \otimes a)) w_n^* \|
\\
& \hspace{2em} {\mbox{}}
 + \| w_n \beta_g (\varphi (1 \otimes a)) w_n^* - \alpha_g (a) \|
\\
& = \| w_n^* a w_n - \varphi (1 \otimes a ) \|
\\
& \hspace{2em} {\mbox{}}
 + \| w_n \varphi (1 \otimes \alpha_g (a)) w_n^* - \alpha_g (a) \|.
\end{align*}

Applying~(\ref{Eq_5Y06_Lim0}) to $a$ and $\af_g (a)$,
we find that
$\displaystyle{
 \lim_{n \to \infty} \| v_n \beta_g (a) v_n^* - \alpha_g (a) \| = 0}$.
This proves the claim.

By Theorem 3.5 in~\cite{MI},
there exists an approximately inner
automorphism $\theta$ such that
$\theta \circ \alpha_g \circ \theta^{-1} = \beta_g$ for every
$g \in G$.
Part~(\ref{P_2_6_Stab}) follows.

Part~(\ref{P_2_6_CRP})
is now immediate from Lemma~\ref{L:PTypeCRP}
and Lemma~\ref{L_5Y02_CtRkTP}.
\end{proof}

\section{The construction}\label{Sec_Constr}

In this section we describe a method to construct simple
separable \ca{s} not isomorphic to their opposite
algebras.
We also show that these \ca{s} satisfy
the Universal Coefficient Theorem.
Throughout this section $q$ is a fixed integer with $q \geq 2$.
The construction is a generalization
of the construction of~\cite{PhV} for $q = 3$.
In Section~\ref{Sec_Main},
we will restrict $q$ to being
an odd prime such that $- 1$
is not a square mod~$q$.

\begin{definition}\label{D:2.1}
Let $q \in \{ 2, 3, \ldots \}$.
Define the \ca~$A_q$ to be the reduced free product
of $q$~copies of $C ([0, 1])$ and the \ca~${\mathbb{C}}^q$,
amalgamated over $\mathbb{C}$, taken with respect to
the states given by Lebesgue measure $\mu$ on each
copy of $C ([0, 1])$ and the state given by
\[
\om (c_1, c_2, \ldots, c_q)
 = \frac{1}{q} \big( c_1 + c_2 + \cdots + c_q \big)
\]
on $\mathbb{C}^q$.
That is,
\[
A_q = C ([0, 1]) \star_{\mathrm{r}} C ([0, 1]) \star_{\mathrm{r}}
 \cdots \star_{\mathrm{r}} C ([0, 1])
 \star_{\mathrm{r}} {\mathbb{C}}^q.
\]

For $k = 1, 2, \ldots, q$ we denote by
$\ep_k \colon C ([0, 1]) \to A_q$
the inclusion of the $k$-th copy of $C ([0, 1])$ in~$A_q$.
Set
\begin{equation}\label{equ2.1}
v = \big( e^{2 \pi i/q}, \, 1, \, e^{2 (q - 1) \pi i/q},
       \, e^{2 (q - 2) \pi i/q}, \, \ldots, \, e^{4 \pi i/q} \big)
 \in {\mathbb{C}}^q,
\end{equation}
and regard $v$ as a unitary in $A_q$ via the obvious
inclusion.
\end{definition}

\begin{lemma}\label{L_3702_ExAf}
There exists a unique automorphism
$\alpha \in \Aut (A_q)$
such that for all $f \in C ([0, 1])$ we have
\begin{equation}\label{Eq:Af1}
\alpha (\ep_{1} (f)) = \ep_{2} (f),
\quad
\alpha (\ep_{2} (f)) = \ep_{3} (f),
\quad
\ldots,
\quad
\alpha (\ep_{q - 1} (f)) = \ep_{q} (f),
\end{equation}
\begin{equation}\label{Eq:Af2}
\alpha (\ep_{q} (f)) = \Ad (v) (\ep_{1} (f)),
\end{equation}
and
\begin{equation}\label{Eq:Af3}
\alpha (v) = e^{ - 2 \pi i / q} v.
\end{equation}
Moreover, with $v$ as in~(\ref{equ2.1}), we have $\alpha^q = \Ad (v)$.
\end{lemma}

\begin{proof}
The proof is the same as that of Lemma~4.6 of~\cite{PhV}.
\end{proof}

\begin{remark}\label{R:2.2}
The \ca{} $A_q$ is unital, separable, simple, exact,
and has a unique tracial state.
Exactness follows from Theorem~3.2 of~\cite{Dykema}.
Simplicity and uniqueness of the tracial state follow
by applying the corollary on page 431 of~\cite{Av} several
times.
Lastly, simplicity and the existence of a faithful
tracial state imply that $A_q$ is stably finite.
\end{remark}

We also need an action on a UHF algebra.
We use a different model than in~\cite{PhV},
which has the advantage that the computation of the Connes invariant
is more explicit.
We start with some notation
and a preliminary lemma which we isolate from the main argument
for convenience.

\begin{notation}\label{D:2.3}
Let $d \in \N$.
Define $\varphi_n \colon M_{d^n} \to M_{d^{n + 1}}$
by $\varphi_n (x) = \operatorname{diag} (x, x, \ldots, x)$
for $x \in M_{d^n}$.
Denote by $B_d$ the UHF algebra obtained as the direct limit
of the system $(M_{d^n}, \varphi_n)_{n \in \N}$.
We identify $M_{d^n}$ with $\displaystyle{\bigotimes_{k = 1}^n M_d}$
and $B_d$ with $\displaystyle{\bigotimes_{k = 1}^{\infty} M_d}$.
\end{notation}

\begin{lem}\label{L_2808_Sim}
Let $n \in \N$, let $(e_{j, k})_{j, k = 1}^{n}$
be the standard system of matrix units in $M_{n}$,
and let $s \in M_n$ be the shift unitary
$\displaystyle{s = e_{1, n} + \sum_{j = 2}^{n} e_{j, j - 1}}$.
Let $v \in M_n$ be a diagonal unitary.
Then there are $\gm \in {\mathbb{T}}$ and a diagonal unitary $y \in M_n$
such that $y v s y^* = \gm s$.
\end{lem}

\begin{proof}
Write $v = \diag \bigl( \bt_1, \bt_2, \ldots, \bt_n \bigr)$
with $\bt_1, \bt_2, \ldots, \bt_n \in {\mathbb{T}}$.
Choose $\gm \in {\mathbb{T}}$ such that
$\gm^n = {\displaystyle{ \prod_{j = 1}^n \bt_j }}$.
Define $\ld_j = {\overline{\gm}} \bt_j$ for $j = 1, 2, \ldots, n$.
Then $w = \diag \bigl( \ld_1, \ld_2, \ldots, \ld_n \bigr)$
satisfies $w = {\overline{\gm}} v$.
Moreover,
\begin{equation}\label{Eq_2812_Prd}
\prod_{j = 1}^n \ld_j = 1.
\end{equation}
Define
\[
\zt_1 = {\overline{\ld_1}},
\qquad
\zt_2 = {\overline{\ld_1 \ld_2}},
\qquad
\ldots
\qquad
\zt_{n - 1} = {\overline{\ld_1 \ld_2 \cdots \ld_{n - 1}}},
\qquad
\zt_n = 1,
\]
and $y = \diag \bigl( \zt_1, \zt_2, \ldots, \zt_n \bigr)$.
Using (\ref{Eq_2812_Prd}), one checks that $y s y^* s^* = w^*$.
So $y s y^* s^* = \gm v^*$.
Since $y$ commutes with $v$, this gives $y v s y^* = \gm s$.
\end{proof}

%%%%%%%%%%%%%%%%%%%%%%%%%%%%%%%%%%%%%%%%%%%%%%%%%%%%%%%%%%%%%%%%%%%%%%%%
%%%%%%%%%%%%%%%%%%%%%%%%%%%%%%%%%%%%%%%%%%%%%%%%%%%%%%%%%%%%%%%%%%%%%%%%

In the next lemma, $B_{q^2}$ and $B_q$ are of course isomorphic.
But we find it convenient to notationally distinguish them:
$B_{q^2}$ is the algebra in which we carry out the construction,
and $B_q$ is the subalgebra in~(\ref{I_modelBq}).

\begin{lemma}\label{UHFmodel}
Let $q \in \N$ satisfy $q \geq 2$.
There exist unitaries $g \in B_{q^2}$ and
$u_k \in B_{q^2}$ for $k = - 1, 0, 1, 2, \ldots$
such that, taking $E_{ - 1} = C^* ( g )$ and
$E_{n} = C^* ( g, u_0, u_1, \ldots, u_n )$ for $n \in \Nz$,
and using the unique \tst{} on~$B_{q^2}$,
the following hold:
\begin{enumerate}
\item\label{I_modelB_uk}
For $k = - 1, 0, 1, 2, \ldots$,
$\spec (u_k) = \{ \zt \in {\mathbb{T}} \colon \zt^{q^2} = 1 \}$,
and for every $\zt \in \spec (u_k)$,
the corresponding spectral \pj{} of $u_k$ has trace $1 / q^2$.
\item\label{I_modelB_g}
$\spec (g) = \{ \zt \in {\mathbb{T}} \colon \zt^{q} = 1 \}$,
and for every $\zt \in \spec (g)$,
the corresponding spectral \pj{} of $g$ has trace $1 / q$.
\item\label{I_2811_vkg_sm}
$u_k g u_k^* = e^{2 \pi i / q} g$ for $k \in \{ - 1, 0 \}$.
\item\label{I_2811_vkg_lg}
$u_k g u_k^* = g$ for $k = 1, 2, \ldots$.
\item\label{I_modelB_vk_vkp1}
$u_k u_{k + 1} u_k^* = e^{ - 2 \pi i / q^2} u_{k + 1}$
for $k = - 1, 0, 1, 2, \ldots$.
\item\label{I_modelB_ujvk}
$u_j u_k = u_k u_j$ for $j, k = - 1, 0, 1, 2, \ldots$
with $| j - k | \geq 2$.
\item\label{l_modelblock}
$E_n \cong (M_{q^{n + 1}})^q$ for $n = - 1, 0, 1, 2, \ldots$.
\item\label{I_2811_Center}
If $n \in \{- 1, 0, 1, 2, \ldots\}$ is odd,
then the center $Z (E_n)$ is generated by the order~$q$ unitary
$w_{n} = g^* (u_1^*)^q (u_3^*)^q \cdots (u_n^*)^q$
(so $w_{- 1} = g^*$),
and if $n$ is even then $Z (E_n)$ is generated by the order~$q$ unitary
$w_{n} = u_0^q u_2^q u_4^q \cdots u_n^q$.
\item\label{I_modelBq}
$C^* ( g, u_0, u_1, u_2, \ldots) \cong B_q$.

\setcounter{TmpEnumi}{\value{enumi}}
\end{enumerate}
\end{lemma}

\begin{proof}
For $k = 1, 2, \ldots$ let $\sigma_k \colon M_{q^2} \to B_{q^2}$ be the map
\[
\sigma_k (x) = 1 \otimes \cdots \otimes 1 \otimes x
 \otimes 1 \otimes 1 \otimes \cdots,
\]
with $x$ in position~$k$.
(Thus, $\sigma_1 (x) = x \otimes 1 \otimes 1 \otimes \cdots$.)
Let $\rho \colon B_{q^2} \to B_{q^2}$
denote the shift endomorphism of~$B_{q^2}$,
determined by $\rho (\sigma_k (x)) = \sigma_{k + 1} (x)$
for $k \in \N$ and $x \in M_{q^2}$.

Denote by $(e_{j, k})_{j, k = 1}^{q^2}$
the standard system of matrix units in $M_{q^2}$.
Define unitaries $z, s \in M_{q^2}$ by

\[
z = \diag \bigl( 1, \, e^{2 \pi i/q^2},
        \, \ldots, \, e^{2 (q^2 - 1) \pi i/q^2} \bigr)
\andeqn
s = e_{1, q^2} + \sum_{j = 2}^{q^2} e_{j, j - 1}.
\]
Apply Lemma~\ref{L_2808_Sim} with $n = q^2$, with $s$ as given,
and with $v = z^*$, getting $\gm \in {\mathbb{T}}$.
Define unitaries in $B_{q^2}$ by
\[
g = (\gm z^* s)^q \otimes 1 \otimes 1 \otimes \cdots,
\]
and, for $k = 0, 1, \ldots$, 
\[
u_{2 k - 1} = \sigma_{k + 1} (s)
\andeqn
u_{2 k} = \sigma_{k + 1} (z) \sigma_{k + 2} (z^*).
\]
Thus, for example,
\[
u_{- 1} = s \otimes 1 \otimes 1 \otimes \cdots
\andeqn
u_0 = z \otimes z^* \otimes 1 \otimes 1 \otimes \cdots.
\]
Since $s$ is unitarily equivalent to~$z$,
the choice of $\gm$ using Lemma~\ref{L_2808_Sim} implies
that $(\gm z^* s)^q$ is unitarily equivalent to~$z^q$.
So~(\ref{I_modelB_g}) holds if one uses the normalized trace
on $M_{q^2}$.
Therefore~(\ref{I_modelB_g}) holds using the tracial state on~$B_{q^2}$.
Similar reasoning shows that if $k$ is odd, then~(\ref{I_modelB_uk}) holds.
One readily checks that
$z \otimes z^* \in M_{q^4}$ satisfies~(\ref{I_modelB_uk})
if one uses the normalized trace on $M_{q^4}$,
so~(\ref{I_modelB_uk}) for even $k$ follows in the same way.

A direct check shows that $s z s^* = e^{- 2 \pi i/q^2} z$.
Using this, it is easy to verify the relations
(\ref{I_2811_vkg_sm}), (\ref{I_2811_vkg_lg}),
(\ref{I_modelB_vk_vkp1}), and (\ref{I_modelB_ujvk}) of the statement.

We claim that for $n = -1, 0, 1, \ldots$, we have
\begin{equation}\label{Eq_2812_En_Span}
\begin{split}
E_n & = \spn \bigl( \bigl\{ g^k u_0^{l_0} u_1^{l_1} \cdots u_n^{l_n} \colon
  \\
& \hspace*{3em} {\mbox{}}
{\mbox{$0 \leq k \leq q - 1$ and $0 \leq l_0, l_1, \ldots, l_n \leq q^2 - 1$}}
   \bigr\} \bigr).
\end{split}
\end{equation}

The claim follows from $g^q = 1$ and $u_n^{q^2} = 1$
(by (\ref{I_modelB_g}) and (\ref{I_modelB_uk})), and because,
given this, (\ref{I_2811_vkg_sm}), (\ref{I_2811_vkg_lg}),
(\ref{I_modelB_vk_vkp1}), and (\ref{I_modelB_ujvk}) show that
the product of any two of the elements listed in~(\ref{Eq_2812_En_Span})
is a scalar multiple of another of them.

Let $m \in \{ -1, 0, 1, \ldots \}$,
let $n \in \Nz$, and let $k, l_0, l_1, \ldots, l_m \in \Z$.
We compute $u_n ( g^k u_0^{l_0} u_1^{l_1} \cdots u_m^{l_m} ) u_n^*$
using (\ref{I_2811_vkg_sm}), (\ref{I_2811_vkg_lg}),
(\ref{I_modelB_vk_vkp1}), and (\ref{I_modelB_ujvk}).
To this end observe that, by (\ref{I_modelB_vk_vkp1}),
we have $u_k u_{k - 1} u_k^* = e^{2 \pi i / q^2} u_{k - 1}$
for all $k \in \Nz$.
For $m \geq 1$, we then get
\begin{equation}\label{Eq_2812_c_un_gum}
\begin{split}
& u_n ( g^k u_0^{l_0} u_1^{l_1} \cdots u_m^{l_m} ) u_n^*
\\
& \hspace*{1em} {\mbox{}}
 = \begin{cases}
   e^{2 \pi i [k q - l_1] / q^2}
       \cdot g^k u_0^{l_0} u_1^{l_1} \cdots u_m^{l_m}
    & \hspace*{1em} n = 0
        \\
  e^{2 \pi i [l_{n - 1} - l_{n + 1}] / q^2}
      \cdot g^k u_0^{l_0} u_1^{l_1} \cdots u_m^{l_m}
    & \hspace*{1em} n = 1, 2, \ldots, m - 1
        \\
   e^{2 \pi i l_{n - 1} / q^2}
     \cdot g^k u_0^{l_0} u_1^{l_1} \cdots u_m^{l_m}
    & \hspace*{1em} n = m, m + 1
        \\
   g^k u_0^{l_0} u_1^{l_1} \cdots u_m^{l_m}
    & \hspace*{1em} n = m + 2, m + 3, \ldots.
\end{cases}
\end{split}
\end{equation}
The case $m = - 1$ is
\begin{equation}\label{Eq_2812_c_un_g}
u_n g^k u_n^*
  = \begin{cases}
   e^{2 \pi i k / q} g^k & \hspace*{1em} n = 0
        \\
   g^k & \hspace*{1em} n = 1, 2, \ldots,
\end{cases}
\end{equation}
and the case $m = 0$ is
\begin{equation}\label{Eq_2812_c_u_n_gu0}
u_n ( g^k u_0^{l_0}) u_n^*
 = \begin{cases}
   e^{2 \pi i [k q - l_1] / q^2} \cdot g^k u_0^{l_0} & \hspace*{1em} n = 0
        \\
    e^{2 \pi i l_{0} / q^2}
      \cdot g^k u_0^{l_0} & \hspace*{1em} n = 1
       \\
   g^k u_0^{l_0} & \hspace*{1em} n = 2, 3, \ldots.
\end{cases}
\end{equation}
Also,
\begin{equation}\label{Eq_2812_c_ggum}
g ( g^k u_0^{l_0} u_1^{l_1} \cdots u_m^{l_m} ) g^*
 = e^{- 2 \pi i l_0 / q} \cdot g^k u_0^{l_0} u_1^{l_1} \cdots u_m^{l_m}.
\end{equation}

For $n = -1, 0, 1, \ldots$, let $w_m$ be as in~(\ref{I_2811_Center}).
Then $w_m$ is a unitary in $E_m$.
By (\ref{I_2811_vkg_lg}) and (\ref{I_modelB_ujvk}),
for both even and odd~$m$, the factors in the definition of $w_m$ commute.
Therefore $w_m^q = 1$ by (\ref{I_modelB_uk}) and~(\ref{I_modelB_g}).
The relations (\ref{Eq_2812_c_un_gum}), (\ref{Eq_2812_c_un_g}),
(\ref{Eq_2812_c_u_n_gu0}), and~(\ref{Eq_2812_c_ggum})
show that $w_m$ commutes with all the generators of $E_m$.
Therefore $w_m$ is in the center $Z (E_m)$.
These relations also show that
\begin{equation}\label{Eq_2812_Cj_unp1}
u_{m + 1} w_m u_{m + 1}^*
 = \begin{cases}
   e^{- 2 \pi i/q} w_m & \hspace*{1em} {\mbox{$m$ is odd}}
       \\
   e^{2 \pi i/q} w_m & \hspace*{1em} {\mbox{$m$ is even}}.
\end{cases}
\end{equation}
Combined with $w_m^q = 1$, this shows that
$\spec (w_m) = \{ \zt \in {\mathbb{T}} \colon \zt^{q} = 1 \}$.
For $m \in \{ -1, 0, 1, \ldots \}$ and $j \in \{ 0, 1, \ldots, q - 1 \}$,
we now let $p_{m, j}$ be the spectral \pj{} of $w_m$ corresponding
to the eigenvalue $e^{2 \pi i j/q}$.
For fixed $m$, the \pj{s} $p_{m, j}$
are then all unitarily equivalent to each other in $B_{q^2}$,
central in $E_m$, and sum to~$1$.
In particular, they are all nonzero.
It follows that
\begin{equation}\label{Eq_2812_EmSum}
E_m = \bigoplus_{j = 1}^q p_{m, j} E_{m} p_{m, j}.
\end{equation}

For $k \in \Nz$ define
\[
D_k = C^* (u_0 u_2 \cdots u_{2 k}, \, u_{2 k + 1} ).
\]

We claim that $D_k \cong M_{q^2}$.
To see this, use (\ref{Eq_2812_c_un_gum})
(use~(\ref{Eq_2812_c_u_n_gu0}) if $k = 0$) to see that
\[
u_{2 k + 1} ( u_0 u_2 \cdots u_{2 k} ) u_{2 k + 1}^*
 = e^{2 \pi i/q^2} u_0 u_2 \cdots u_{2 k}.
\]
Also, $u_{2 k + 1}^{q^2} = 1$ by (\ref{I_modelB_uk}),
and, using (\ref{I_modelB_uk}) and (\ref{I_modelB_ujvk}),
one checks that $( u_0 u_2 \cdots u_{2 k} )^{q^2} = 1$.
It is easily seen that the universal \ca{} $C$ generated by unitaries
$v$ and~$w$ satisfying
\[
v^{q^2} = w^{q^2} = 1 \andeqn v w v^* =  e^{2 \pi i / q^2} w
\]
is the transformation group \ca{}
of the action of $\Z_{q^2}$ on $\Z_{q^2}$ by translation,
which is isomorphic to~$M_{q^2}$.
Since this algebra is simple, the claim follows.
% QQQ

We claim that if $k \neq l$, then $D_k$ commutes with $D_l$.
For the proof, \wolog{} $k < l$.
It suffices to prove the following:
\begin{enumerate}
% [label=$\mathrm{(\arabic*)}$]
\setcounter{enumi}{\value{TmpEnumi}}
\item\label{I_2812_Comm_2k2l}
$u_0 u_2 \cdots u_{2 k}$ commutes with $u_0 u_2 \cdots u_{2 l}$.
\item\label{I_2812_Comm_2klp1}
$u_0 u_2 \cdots u_{2 k}$ commutes with $u_{2 l + 1}$.
\item\label{I_2812_Comm_kp12l}
$u_{2 k + 1}$ commutes with $u_0 u_2 \cdots u_{2 l}$.
\item\label{I_2812_Comm_kp1lp1}
$u_{2 k + 1}$ commutes with $u_{2 l + 1}$.
\end{enumerate}
Of these, (\ref{I_2812_Comm_2k2l}), (\ref{I_2812_Comm_2klp1}),
and~(\ref{I_2812_Comm_kp1lp1})
all follow from~(\ref{I_modelB_ujvk}),
and (\ref{I_2812_Comm_kp12l}) follows from~(\ref{Eq_2812_c_un_gum}).
The claim is proved.

Let $k \in \Nz$.
Since $D_0, D_1, \ldots, D_k \subseteq E_{2 k + 1}$,
it follows that there is a unital \hm{} from
$D_0 \otimes D_1 \otimes \cdots \otimes D_k \cong M_{q^{2 k + 2}}$
to $E_{2 k + 1}$.
For $j = 0, 1, \ldots, q - 1$, recalling that $p_{2 k + 1, j} \neq 0$
and using~(\ref{Eq_2812_EmSum}),
we see that there is a nonzero unital \hm{}
$M_{q^{2 k + 2}}\to p_{2 k + 1, j} E_{2 k + 1} p_{2 k + 1, j}$.
Therefore $\dim (E_{2 k + 1}) \geq q^{4 k + 5}$.
It follows from~(\ref{Eq_2812_En_Span}) that
$\dim (E_{2 k + 1}) \leq q^{4 k + 5}$.
Therefore $\dim (E_{2 k + 1}) = q^{4 k + 5}$
and the maps $M_{q^{2 k + 2}}\to p_{2 k + 1, j} E_{2 k + 1} p_{2 k + 1, j}$
are isomorphisms, that is,
$E_{2 k + 1} \cong \bigoplus_{j = 1}^q M_{q^{2 k + 2}}$.
This is part~(\ref{l_modelblock}) of the conclusion for odd~$n$,
and implies part~(\ref{I_2811_Center}) of the conclusion for odd~$n$.

{}From~(\ref{Eq_2812_En_Span}) we also get
$\dim (E_{m + 1}) \leq q^2 \dim (E_{m})$ for $m = -1, 0, 1, \ldots$.
We have just seen that $\dim (E_{m}) = q^{2 m + 3}$ when $m$ is odd,
so this relation holds for all $m \in \{ -1, 0, 1, \ldots \}$.
In particular,
the elements listed on the right in (\ref{Eq_2812_En_Span})
are actually a basis for $E_m$.

This fact, and the form of the relations
(\ref{Eq_2812_c_un_gum}), (\ref{Eq_2812_c_un_g}),
(\ref{Eq_2812_c_u_n_gu0}), and~(\ref{Eq_2812_c_ggum}),
implies that $Z (E_m)$ is the linear span of the
elements $g^k u_0^{l_0} u_1^{l_1} \cdots u_m^{l_m}$
with $0 \leq k \leq q - 1$ and $0 \leq l_0, l_1, \ldots, l_m \leq q^2 - 1$
which commute with all of $g, u_0, u_1, \ldots, u_{m}$.
Moreover, these relations imply that this happens exactly when
$l_0 \in \{ 0, q, 2 q, \ldots, q^2 - q \}$, and
\[
k q = l_1, \quad l_0 = l_2, \quad l_1 = l_3, 
\quad \ldots, \quad l_{m - 2} = l_{m},
\quad {\mbox{and}} \quad l_{m - 1} = 0.
\]
If $m$ is even, this says $l_j = 0$ for all odd~$j$
and there is $r \in \{ 0, 1, \ldots, q - 1 \}$
such that $l_j = r q$ for all even~$j$.
Therefore $\dim (Z (E_m)) = q$.
It follows that
\[
Z (E_m) = \spn \bigl( p_{m, 0}, p_{m, 1}, \ldots, p_{m, q - 1} \bigr)
        = C^* (w_m),
\]
just as we saw above for $m$ odd.
This implies part~(\ref{I_2811_Center}) of the conclusion for even~$n$.
Conjugation by $u_{m + 1}$ permutes the \pj{s} $p_{m, j}$ cyclically
(by~(\ref{Eq_2812_Cj_unp1}))
and is an automorphism of $E_m$
(by (\ref{Eq_2812_c_un_gum}), (\ref{Eq_2812_c_un_g}),
and (\ref{Eq_2812_c_u_n_gu0})),
so the summands $p_{m, j} E_{m} p_{m, j}$ are all isomorphic.
Since they are simple, a dimension count shows that
they are all isomorphic to $M_{q^{m + 1}}$.
This is part~(\ref{l_modelblock}) of the conclusion for even~$n$.

It remains to prove~(\ref{I_modelBq}).
Let $m \in \{ -1, 0, 1, \ldots \}$,
and let $\mu = (\mu_{j, k})_{j, k = 1}^{q}$
be the matrix of partial embedding multiplicities of the inclusions
\[
\ps_{j, k} \colon
 p_{m, j} E_m p_{m, j} \to p_{m + 1, k} E_{m + 1} p_{m + 1, k}.
\]
We claim that $\mu_{j, k} = 1$ for $j, k = 1, 2, \ldots, q$.
To prove this, observe that conjugation by $u_{m + 2}$
permutes the \pj{s} $p_{m + 1, k}$ cyclically (by~(\ref{Eq_2812_Cj_unp1}))
but is the identity on $E_m$
(by (\ref{Eq_2812_c_un_gum}), (\ref{Eq_2812_c_un_g}),
and (\ref{Eq_2812_c_u_n_gu0})).
Therefore, for fixed~$j$,
the \pj{s} $p_{m, j} p_{m + 1, k}$ are all unitarily equivalent in~$B_{q^2}$.
Also, $u_{m + 1} \in E_{m + 1}$ and $w_{m + 1} \in Z (E_{m + 1})$,
so $u_{m + 1}$ commutes with the \pj{s} $p_{m + 1, k}$,
while $u_{m + 1}$ permutes the \pj{s} $p_{m, j}$ cyclically
by~(\ref{Eq_2812_Cj_unp1}).
It follows that the \pj{s} $p_{m, j} p_{m + 1, k}$
are all unitarily equivalent in~$B_{q^2}$ for $j, k = 0, 1, \ldots, q - 1$.
Therefore $p_{m, j} p_{m + 1, k} \neq 0$ for all $j$ and~$k$.
Thus $\mu_{j, k} \geq 1$.
Also,
\[
\mu \cdot ( q^{m + 1}, \, q^{m + 1}, \, \ldots, \, q^{m + 1} )
 = ( q^{m + 2}, \, q^{m + 2}, \, \ldots, \, q^{m + 2} ).
\]
For $j = 0, 1, \ldots, q - 1$ we therefore have
$\sum_{k = 0}^{q - 1} \mu_{j, k} = q$.
Since $\mu_{j, k} \geq 1$ for all $j$ and~$k$, the claim follows.

It is now easy to check,
for example by computing the $K_0$-group, that
${\displaystyle{ {\overline{ \bigcup_{m = -1}^{\I} E_m}} \cong B_q}}$.
This is~(\ref{I_modelBq}).
\end{proof}

\begin{lemma}\label{L_3702_ExBt}
Let $q \in \N$ satisfy $q \geq 2$.
Let $g, u_{-1}, u_0, u_1, u_2, \ldots \in B_{q^2}$
satisfy the conditions of Lemma~\ref{UHFmodel}.
Set $E = C^* ( g, u_0, u_1, u_2, \ldots) \subseteq B_{q^2}$.
Then $\beta = \Ad (u_{ - 1}) |_E$ is an automorphism of~$E$
such that:
\begin{enumerate}
% [label=$\mathrm{(\arabic*)}$]
%
\item\label{I_2816_btq}
$\beta^q = \Ad (g)$.
\item\label{I_2816_bt_g}
$\beta (g) = e^{2 \pi i/q} g$.
\item\label{I_2816_u0}
$\beta (u_0) = e^{- 2 \pi i/q^2} u_0$.
\item\label{I_2816_oth_un}
$\beta (u_n) = u_n$ for $n \in \N$.
\item\label{I_2816_out}
$\bt^k$ is outer for $k = 1, 2, \ldots, q - 1$.
\end{enumerate}
\end{lemma}

\begin{proof}
It follows from (\ref{I_2811_vkg_sm}),
(\ref{I_modelB_vk_vkp1}), and~(\ref{I_modelB_ujvk}) of Lemma~\ref{UHFmodel}
that $u_{-1} E u_{-1}^* = E$,
so that $\bt$ is an automorphism,
and that parts~(\ref{I_2816_bt_g}), (\ref{I_2816_u0}),
 and~(\ref{I_2816_oth_un})
of the conclusion hold.
{}From Lemma \ref{UHFmodel}(\ref{I_2811_vkg_sm})
we get $g u_0 g^* = e^{- 2 \pi i / q}u_0$.
This, parts~(\ref{I_2816_bt_g}), (\ref{I_2816_u0}), and~(\ref{I_2816_oth_un}),
and the relations (\ref{I_2811_vkg_lg}),
(\ref{I_modelB_vk_vkp1}), and (\ref{I_modelB_ujvk}) in Lemma \ref{UHFmodel},
imply~(\ref{I_2816_btq}).

It remains to prove~(\ref{I_2816_out}).
Let $k \in \{ 1, 2, \ldots, q - 1\}$.
For $n \in \Nz$, let
$w_{2 n} = u_0^q u_2^q u_4^q \cdots u_{2 n}^q$,
as in Lemma \ref{UHFmodel}(\ref{I_2811_Center}).
Using Lemma \ref{UHFmodel}(\ref{I_modelB_vk_vkp1})
and (\ref{I_modelB_ujvk}),
we get $\bt^k (w_{2 n}) = e^{- 2 \pi i k / q} w_{2 n}$, so
$\limi{n} \| \bt^k (w_{2 n}) - w_{2 n} \|
  = | 1 - e^{- 2 \pi i k / q} | \neq 0$.
However, by Lemma \ref{UHFmodel}(\ref{I_2811_Center}),
for any unitary $y \in E$,
we have $\limi{n} \| y w_{2 n} y^* - w_{2 n} \| = 0$.
So $\bt^k \neq \Ad (y)$.
\end{proof}

\begin{lemma}\label{R:2.5}
Let $q \in \N$ satisfy $q \geq 2$.
Let $A_q$ and $v$ be as in Definition~\ref{D:2.1},
and let $\af \in \Aut (A_q)$ be as in Lemma~\ref{L_3702_ExAf}.
Let $g, u_0, u_1, u_2, \ldots$ be as in Lemma~\ref{UHFmodel}
and, following Lemma \ref{UHFmodel}(\ref{I_modelBq}),
identify $B_q$ with the algebra $E = C^* ( g, u_0, u_1, u_2, \ldots)$.
Let $\bt \in \Aut (B_q)$ correspond to the automorphism
$\bt \in \Aut (E)$ of Lemma~\ref{L_3702_ExBt}.
There exists a unitary
$w \in C^* (v \otimes g) \subseteq A_q \otimes B_q$
with the following properties:
\begin{enumerate}
% [label=$\mathrm{(\arabic*)}$]
%
\item\label{I_2813_R_2_5_P}
$w^q = (v \otimes g)^*$.
\item\label{I_2813_R_2_5_Fix}
$(\alpha \otimes \beta) (w) = w$.
\item\label{I_2813_R_2_5_Ct}
$w$ commutes with $1 \otimes (u_0^*)^q g$.
\item\label{I_2813_R_2_5_Rkh}
If we set
\[
\gamma = \Ad (w) \circ (\alpha \otimes \beta) \in \Aut ( A_q \otimes B_q),
\]
then $\gamma$ generates an action of~${\mathbb{Z}}_q$
which has the Rokhlin property.
\setcounter{TmpEnumi}{\value{enumi}}
\end{enumerate}
\end{lemma}

\begin{proof}
The construction of $w$ satisfying (\ref{I_2813_R_2_5_P})
and~(\ref{I_2813_R_2_5_Fix}) is the same as in Lemma~4.8 of~\cite{PhV}.
It follows from Lemma \ref{UHFmodel}(\ref{I_2811_vkg_sm})
that $v \otimes g$ commutes with $1 \otimes (u_0^*)^q g$,
so~(\ref{I_2813_R_2_5_Ct}) follows from $w \in C^* (v \otimes g)$.
It is straightforward to show that
$\big[ \Ad (w) \circ (\alpha \otimes \beta) \big]^q = \id_{A_q \otimes B_q}$.

Next, we claim that for any $\varepsilon > 0$ and any finite subset
$F \subseteq A_q \otimes B_q$,
there are projections $e_0, e_1, \ldots, e_{q - 1} \in A_q \otimes B_q$
such that:
\begin{enumerate}\label{Rokhlin}
\setcounter{enumi}{\value{TmpEnumi}}
\item\label{I_2814_Perm}
$\| (\alpha \otimes \beta) (e_k) - e_{k + 1}\| < \varepsilon$
for $k = 0, 1, \ldots, q - 2$
and $\|(\alpha \otimes \beta) (e_{q - 1}) - e_0 \| < \varepsilon$.
\item\label{I_2814_Comm}
$\| y e_k - e_k y \| < \varepsilon$ for $k = 0, 1, \ldots, q - 1$
and all $y \in F$.
\item\label{I_2814_Sum}
${\displaystyle{\sum_{k = 0}^{q - 1} e_k = 1}}$.
\end{enumerate}

To prove the claim, for $n \in \Nz$
let $E_n = C^* (g, u_0, u_1, \ldots, u_n)$, as in Lemma \ref{UHFmodel}.
By a standard approximation argument,
it is enough to consider finite subsets of $A_q \otimes B_q$ of the form
\[
F = \{a \otimes b \colon {\mbox{$a \in S$ and $b \in T$}} \}
\]
for $n \in \Nz$ and finite subsets
$S \subseteq A_q$ and $T \subseteq E_{2 n}$.
Let $\varepsilon > 0$.
Let $w_{2 n} = u_0^q u_2^q u_4^q \cdots u_{2 n}^q$,
as in Lemma \ref{UHFmodel}(\ref{I_2811_Center}).
By Lemma \ref{UHFmodel}(\ref{I_2811_vkg_lg}) and~(\ref{I_modelB_vk_vkp1}),
and the formula for $\bt$ in Lemma~\ref{L_3702_ExBt},
\begin{equation}\label{Eq_2814_btw}
\bt (w_{2 n}) = e^{- 2 \pi i / q} w_{2 n}.
\end{equation}
Let $p_0, p_1, \ldots, p_{q - 1} \in E_{2 n}$ be the spectral \pj{s}
for $w_{2 n}$, labelled so that
${w_{2 n} = \displaystyle{\sum_{k = 0}^{q - 1} e^{2 \pi i k / q} p_k}}$.
It follows from~(\ref{Eq_2814_btw}) that $\bt (p_k) = p_{k + 1}$
for $k = 0, 1, \ldots, q - 2$
and $\bt (p_{q - 1}) = p_0$.
For $k = 0, 1, \ldots, q - 1$, set $e_k = 1 \otimes p_k$.
These \pj{s} satisfy conditions (\ref{I_2814_Perm}) and~(\ref{I_2814_Sum}).
Since $p_k \in Z (E_{2 n})$ (by Lemma \ref{UHFmodel}(\ref{I_2811_Center})),
we have $e_k (a \otimes b) = (a \otimes b) e_k$
for all $a \in S$ and $b \in T$.
So (\ref{I_2814_Comm}) holds.
This proves the claim.

To show that $\gamma$ satisfies the Rokhlin property,
let $F \subseteq A_q \otimes B_q$ be finite and let $\varepsilon > 0$.
Construct projections $e_0, e_1, \ldots, e_{q - 1}$ as in the claim,
with $F \cup \{w\}$ in place of $F$
and $\frac{\varepsilon}{2}$ in place of $\varepsilon$.
We verify the analogs of (\ref{I_2814_Perm}),
(\ref{I_2814_Comm}), and~(\ref{I_2814_Sum})
with $\gamma = \Ad (w) \circ (\alpha \otimes \beta)$ in place of
$\alpha \otimes \beta$.
Only the analog of~(\ref{I_2814_Perm}) requires proof.
For $k = 0, 1, \ldots, q - 2$ we have
\begin{equation*}
\begin{split}
\| \gamma (e_k) - e_{k + 1}\|
& \leq \bigl\| \Ad (w) (\alpha \otimes \beta) (e_k)
        - \Ad (w) (e_{k + 1}) \bigr\|
   + \| w e_{k + 1} w^* - e_{k + 1} \|
 \\
& \leq \| (\alpha \otimes \beta) (e_k) - e_{k + 1}\|
   + \|w e_{k + 1} - e_{k + 1} w \|
  < \frac{\varepsilon}{2} + \frac{\varepsilon}{2}
  = \varepsilon.
\end{split}
\end{equation*}
The proof that $\|(\alpha \otimes \beta) (e_{q - 1}) - e_0 \| < \varepsilon$
is essentially the same.
\end{proof}

\begin{definition}\label{D:2.4}
Let $A_q$ be as in Definition~\ref{D:2.1}
and let $B_q$ be as in Notation~\ref{D:2.3}.
Set $C_q = A_q \otimes B_q$, and let $\gamma$ be
the automorphism of Lemma \ref{R:2.5}.
We also write $\gamma$ for the action
of~${\mathbb{Z}}_q$ generated by this automorphism,
and define the \ca~$D_q$ by
$D_q = C_q \rtimes_{\gamma} {\mathbb{Z}}_q$.
\end{definition}

\begin{proposition}\label{pcaseproposition}
Let $q \in \{ 2, 3, \ldots \}$.
The \ca{} $D_q = C_q \rtimes_{\gamma} {\mathbb{Z}}_q$
of Definition~\ref{D:2.4} is simple, separable, unital, and exact.
It tensorially absorbs the $q^{\infty}$ UHF algebra~$B_q$
and the Jiang-Su algebra~$Z$.
Moreover, $D_q$ is approximately divisible, stably finite,
has real rank zero and stable rank one,
and has a unique tracial state,
which determines the order on projections over~$D_q$.
Also,
\[
K_0 (D_q) \cong {\mathbb{Z}} \big[ \tfrac{1}{q} \big]
\andeqn
K_1 (D_q) = 0,
\]
where the first isomorphism sends $[1]$ to~$1$, and
is an isomorphism of ordered groups.
Finally,
letting ${\mathbb{Z}} \big[ \tfrac{1}{q} \big]_{+}$
be the set of nonnegative elements
in ${\mathbb{Z}} \big[ \tfrac{1}{q} \big] \subseteq \R$,
the Cuntz semigroup of $D_q$ is given by
\[
W (D_q)
 \cong {\mathbb{Z}} \big[ \tfrac{1}{q} \big]_{+} \amalg (0, \infty).
\]
\end{proposition}

\begin{proof}
We first consider the algebra $C_q = A_q \otimes B_q$
in place of~$D_q$, and we prove that it has most of the
properties listed for~$D_q$.
The exceptions are that we
do not prove stable finiteness or that the order on projections
over $C_q$ is determined by traces, the K-theory is different
(and we postpone its calculation),
and we do not compute the Cuntz semigroup.

It is obvious that $C_q$ is separable and unital.
To prove simplicity of $C_q$,
use simplicity of $A_q$
(Remark~\ref{R:2.2}),
simplicity and nuclearity of the UHF algebra~$B_q$,
and the corollary on page~117 of~\cite{Take1}.
(We warn that~\cite{Take1} systematically
refers to tensor products as ``direct products''.)
Exactness of $C_q$ follows from exactness of~$A_q$
(Remark~\ref{R:2.2}),
exactness of~$B_q$,
and Proposition 7.1(iii) of~\cite{Kr2}.
Since $A_q$ and $B_q$ have unique tracial states
(the first by Remark~\ref{R:2.2}),
Corollary~6.13 of~\cite{CP}
(or Lemma~\ref{L_3705_AffHme} below)
implies that $C_q$ has a unique tracial state.
Since $A_q$ is stably finite (Remark \ref{R:2.2}),
and $B_q$ is a UHF algebra, Corollary~6.6
of~\cite{Rordam1} implies that
${\operatorname{tsr}} (A_q \otimes B_q) = 1$.
The algebra $B_q$ is approximately divisible
by Proposition~4.1 of~\cite{BKR},
so $A_q \otimes B_q$ is approximately divisible.
Since $C_q$
is simple, approximately divisible, exact,
and has a unique tracial state,
it has real rank zero by Theorem~1.4(f) of~\cite{BKR}.
The algebra $B_q$ tensorially absorbs~$B_q$,
and tensorially absorbs the Jiang-Su algebra~$Z$
by Corollary 6.3 of~\cite{JS}.
Therefore $C_q$ tensorially absorbs both algebras.

The algebra $D_q$ is separable and unital because $C_q$ is.
Exactness of $D_q$ follows from Proposition 7.1(v) of~\cite{Kr2}.
Parts (\ref{I_2816_btq}) and~(\ref{I_2816_out}) of Lemma~\ref{L_3702_ExBt}
say that $\beta$ has period $q$ in ${\operatorname{Out}} (B_q)$.
So, by Theorem~1 in~\cite{Was},
for $k = 1, 2, \ldots, q - 1$ the automorphism $\gm^k$ is outer.
Theorem~3.1 of~\cite{Ks1} now implies that $D_q$ is simple.
Since $\gm$ has
the Rokhlin property by Lemma \ref{R:2.5}(\ref{I_2813_R_2_5_Rkh}), $D_q$
has a unique tracial state by Proposition 4.14
of~\cite{OsaPhil}, ${\operatorname{tsr}} (D_q) = 1$
by Proposition 4.1(1) of~\cite{OsaPhil}, $D_q$ is
approximately divisible by Proposition 4.5
of~\cite{OsaPhil}, and $D_q$ has real rank zero
by Proposition 4.1(2) of~\cite{OsaPhil}.
Combining
Corollary 3.4(1) of~\cite{HW} with the Rokhlin property,
we see that $D_q$ absorbs both $B_q$ and $Z$.
Simplicity of $D_q$ and existence of a tracial state
imply stable finiteness.

It now follows from Proposition~2.6 of~\cite{PhV}
that the order on projections over $D_q$ is determined
by traces.

The computation of $K_0 (D_q)$ is done in the same
way as in the proof of Proposition~7.2 of~\cite{PhV},
and we refer the reader to that article for the many details
we omit in the following computation.
Here we have
\[
K_0 (A_q) \cong {\mathbb{Z}}^q,
\quad
K_1 (A_q) = 0,
\quad
K_0 (B_q) \cong {\mathbb{Z}} \big[ \tfrac{1}{q} \big],
\quad {\mbox{and}} \quad
K_1 (B_q) = 0,
\]
so that the K\"{u}nneth formula (see~\cite{Sc2}) gives
\[
K_0 (C_q) \cong {\mathbb{Z}} \big[ \tfrac{1}{q} \big]^q
\andeqn
K_1 (C_q) = 0.
\]
Moreover, by the argument used in the proof of
Proposition 7.2 of~\cite{PhV},
\[
K_* (D_q)
 \cong \bigcap_{m = 0}^{q - 1}
  \operatorname{Ker} (\id - K_* (\gamma^m)).
\]

For $j = 1, 2, \ldots, q$, define
$r_j = (0, \ldots, 0, 1, 0, \ldots, 0) \in \C^q$
where $1$ is in the $j$-th position.
Then the unitary $v$ of~(\ref{equ2.1}) is
\[
v = e^{2 \pi i/q} r_1 + r_2 + e^{2 (q - 1) \pi i/q} r_3
   + \cdots + e^{4 \pi i/q} r_q,
\]
and
\[
\alpha (v)
 = r_1 + e^{2 (q - 1) \pi i/q} r_2
   + e^{2 (q - 2) \pi i/q} r_3
   + \cdots + e^{2 \pi i/q} r_q.
\]
This implies that $\alpha (r_j) = r_{j - 1}$
for $j = 2, 3, \ldots, q$
and that $\alpha (r_1) = r_q$.
Since $\Ad (w)$ and $\beta$ are trivial on K-theory,
it follows that
$K_0 (\gamma) \colon {\mathbb{Z}} \big[ \tfrac{1}{q} \big]^q
 \to {\mathbb{Z}} \big[ \tfrac{1}{q} \big]^q$
is given by
\[
K_0 (\gamma) (\et_1, \et_2, \ldots, \et_q) = (\et_2, \et_3, \ldots, \et_{q}, \et_1).
\]
Therefore
$\id - K_0 (\gamma)$ corresponds to the matrix
\[
\left(  \begin{array}{rrrrrr}
   1 & - 1 & 0 & \cdots & 0 & 0 \\
   0 &  1 & - 1 & \cdots & 0 & 0 \\
   0 & 0 & 1 & \cdots & 0 & 0 \\
   \vdots & \vdots & \vdots & \ddots & \vdots & \vdots \\
   0 & 0 & 0 & \cdots & 1 & - 1 \\
  - 1 &  0 & 0 & \cdots & 0 & 1
\end{array} \right).
\]

The map $\et \to (\et, \et, \ldots, \et)$
is an isomorphism from ${\mathbb{Z}} \big[ \tfrac{1}{q} \big]$
to $\operatorname{Ker} (\id - K_0 (\gamma))$,
and one checks that its image is contained in
$\operatorname{Ker} (\id - K_0 (\gamma^m))$ for all $m$ such that
$0 \leq m \leq q - 1$.
Therefore this map is an isomorphism from
${\mathbb{Z}} \big[ \tfrac{1}{q} \big]$ to
$\displaystyle{\bigcap_{m = 0}^{q - 1}
  \operatorname{Ker} (\id - K_0 (\gamma^m))}$.

The computation of the Cuntz semigroup
now follows from Remark~\ref{R:2.15} by observing that
$V (D_q)$ is the positive part of $K_0 (D_q)$
and the uniqueness of the tracial state on $D_q$
implies that
${\operatorname{LAff}}_{\operatorname{b}} (\T (D_q))_{++}
 = (0, \infty)$.
\end{proof}

\begin{proposition}\label{C:2.7}
Let $C_q$  and
$\gamma \colon \mathbb{Z}_q \to \Aut (C_q)$ be
as in Definition~\ref{D:2.4}.
Then $D_q = C_q \rtimes_{\gamma} \mathbb{Z}_q$ satisfies
the Universal Coefficient Theorem.
\end{proposition}

\begin{proof}
By Theorem~1.1 in \cite{Ha2} (see also Theorem 4.1 in \cite{Ge1}),
the algebra $A_q$ in Definition~\ref{D:2.1} is KK-equivalent to
the full free product
\[
E = C ([0, 1]) \star_{\mathbb{C}} \cdots
 \star_{\mathbb{C}} C ([0, 1]) \star_{\mathbb{C}} {\mathbb{C}}^q.
\]
It is shown in the proof of Theorem 2.7 in \cite{Thn2}
that if $A$ and $B$ are separable unital \ca{s},
then the suspension $S (A \star_{\mathbb{C}} B)$
of the amalgamated free product
is KK-equivalent to the mapping cone of the
inclusion $\mathbb{C} \hookrightarrow A \oplus B$.
Therefore $A \star_{\mathbb{C}} B$
satisfies the Universal Coefficient Theorem when $A$ and $B$ do.
Arguing inductively, we see that $E$, and therefore also $A_q$,
satisfies the Universal Coefficient Theorem.
Since $C_q$ is the tensor product of $A_q$ with a UHF algebra,
it too satisfies the Universal Coefficient Theorem.

Observe also that $C_q = A_q \otimes B_q$
absorbs the $q^{\infty}$ UHF algebra.
In addition, by Lemma \ref{R:2.5}(\ref{I_2813_R_2_5_Rkh}),
the action $\gamma$ has the
Rokhlin property, so Proposition \ref{P_2_6}(\ref{P_2_6_CRP})
implies that $\gamma$ has the continuous Rokhlin property.
Using Proposition \ref{UCTcrossedproduct}, we
conclude that $C_q \rtimes_{\gamma} \mathbb{Z}_q$
satisfies the Universal Coefficient Theorem.
\end{proof}

\section{The main step}\label{Sec_Main}

Let $D_q = C_q \rtimes {\mathbb{Z}}_q$
be as in Definition~\ref{D:2.4},
and let $\ta$ be its unique tracial state.
In this section we show that
if $q$ is an odd prime such that $- 1$
is not a square mod~$q$,
then $\pi_{\ta} (D_q)''$
is not isomorphic to its opposite algebra.
This is the main step in proving that $D_q$,
as well as the tensor product $E \otimes D_q$ for suitable~$E$,
is not isomorphic to its opposite algebra.

The following result belongs to the theory of cocycle conjugacy,
but we have not found a reference in the literature.

\begin{lemma}\label{L:Cocycle}
Let $M$ be a factor and let $n \in \N$.
Let $\af, \bt \colon \Z_n \to \Aut (M)$
be actions of $\Z_n$ on~$M$.
Write the elements of $\Z_n$ as $0, 1, \ldots, n - 1$,
so that, for example, the automorphisms generating
the actions are $\af_1$ and $\bt_1$.
Suppose that
there is a unitary $y \in M$ such that
$\bt_1 = \Ad (y) \circ \af_1$.
Then there is an isomorphism
$\ph \colon M \rtimes_{\bt} \Z_n \to M \rtimes_{\af} \Z_n$
which intertwines the dual actions,
that is, for all $l \in {\widehat{\Z_n}}$ we have
$\ph \circ {\widehat{\bt}}_l = {\widehat{\af}}_l \circ \ph$.
\end{lemma}

\begin{proof}
For $k \in \Z$ we write $\af_k = \af_1^k$ and $\bt_k = \bt_1^k$.
(This agrees with the notation
in the statement
when $k \in \{ 0, 1, \ldots, n - 1 \}$.)
For $k \in \N$ define a unitary $y_k \in M$ by
\[
y_k = y \af_1 (y) \af_2 (y) \cdots \af_{k - 1} (y).
\]
Set $y_0 = 1$,
and define
$y_{k} = \af_{k} (y_{- k}^*)$ for $k < 0$.
Then one easily checks that $\Ad (y_k) \circ \af_k = \bt_k$
for all $k \in \Z$,
and moreover that $y_j \af_j (y_k) = y_{j + k}$
for all $j, k \in \Z$.

Since $\af_n = \bt_n = \id_M$
and $M$ is a factor,
we have $y_n \in \mathbb{C} \cdot 1$.
So there is a scalar $\zt$ with $| \zt | = 1$
such that $y_n = \zt^{n} \cdot 1$.
For $k \in \Z$ define
$z_k = \zt^{-k} y_{k}$.
Then $z_k$ is unitary,
and we have $\Ad (z_k) \circ \af_k = \bt_k$
for all $k \in \Z$
and $z_j \af_j (z_k) = z_{j + k}$
for all $j, k \in \Z$.
Moreover, $z_j = z_k$ whenever $n$ divides $j - k$.

Let $u_0, u_1, \ldots, u_{n - 1}$
be the standard unitaries in the crossed product
$M \rtimes_{\alpha} {\mathbb{Z}}_n$ which implement~$\alpha$,
so that for $a \in M \subseteq M \rtimes_{\af} \Z_n$ we have
$u_k a u_k^* = \af_k (a)$
and
\[
M \rtimes_{\af} \Z_n
 = \left\{ \sum_{k = 0}^{n - 1} a_k u_k \colon
    a_0, a_1, \ldots a_{n - 1} \in M \right\}.
\]
Similarly let $v_0, v_1, \ldots, v_{n - 1}$ be the
standard unitaries in $M \rtimes_{\bt} \Z_n$
which implement $\bt$.
Then there is a unique linear bijection
$\ph \colon M \rtimes_{\bt} \Z_n \to M \rtimes_{\af} \Z_n$
such that $\ph (a v_k) = a z_k u_k$ for $a \in M$
and $k = 0, 1, \ldots, n - 1$.
One checks, using the properties of $(z_k)_{k \in \Z}$,
that $\ph$ is a \hm.
Moreover,
\[
\big( \ph \circ {\widehat{\bt}}_l \big) (a v_k)
 = \ph \big( e^{2 \pi i k l / n} a v_k \big)
 = e^{2 \pi i k l / n} a z_k u_k
 = {\widehat{\af}}_l (a z_k u_k)
 = ({\widehat{\af}}_l \circ \ph) (a v_k)
\]
for every $a \in M$ and $k \in \{ 0, 1, \ldots, n - 1 \}$.
Therefore
$\ph \circ {\widehat{\bt}}_l = {\widehat{\af}}_l \circ \ph$.
\end{proof}

\begin{lemma}\label{L_3705_Q}
Let $A$ and $B$ be \ca{s}.
Let $\rh$ be a state on~$A$ and let $\om$ be a state on~$B$.
Then
$\pi_{\rh \otimes \om} (A \otimes_{\mathrm{min}} B) ''
  \cong \pi_{\rh} (A)'' {\overline{\otimes}} \pi_{\om} (B)''$.
\end{lemma}

\begin{proof}
The proof is straightforward
(one starts by identifying $H_{\rh \otimes \om}$
with $H_{\rh} \otimes H_{\om}$),
and is omitted.
\end{proof}

For $q \in \{2, 3, \ldots \}$ let
$D_q = C_q \rtimes_{\gamma} {\mathbb{Z}}_q$ be the \ca{}
of Definition~\ref{D:2.4}.
Our next step is to show that $D_q$ is not isomorphic to its
opposite algebra whenever $q$ is an odd prime
such that $- 1$ is not a square mod~$q$, by associating to
$D_q$ a $\mathrm{II}_1$ factor $T_q$ and
computing the Connes invariant of $T_q$.

For $q \in \{2, 3, \ldots \}$, set
\[
N_q = \big[ \bigstar_1^q L^{\infty} ([0, 1]) \big]
 \star {\mathcal{L}} ( {\mathbb{Z}}_q),
\]
and for $k = 1, 2, \ldots, q$ denote by
${\overline{\ep}}_k \colon L^{\infty} ([0, 1]) \to N_q$
the inclusion of the \mbox{$k$-th} free factor
$L^{\infty} ([0, 1])$ in~$N_q$.
Let $v$ be the element of
$\mathcal{L} ({\mathbb{Z}}_q) \cong \mathbb{C}^q$
defined in~(\ref{equ2.1})
and let ${\overline{\alpha}}$ be the automorphism of $N_q$ given by
\begin{equation}\label{Eq:AfTd1}
{\overline{\alpha}} ({\overline{\ep}}_{1} (f))
 = {\overline{\ep}}_{2} (f),
\quad
{\overline{\alpha}} ({\overline{\ep}}_{2} (f))
 = {\overline{\ep}}_{3} (f),
\quad
\ldots,
\quad
{\overline{\alpha}} ({\overline{\ep}}_{q - 1} (f))
 = {\overline{\ep}}_{q} (f),
\end{equation}
\begin{equation}\label{Eq:AfTd2}
{\overline{\alpha}} ({\overline{\ep}}_{q} (f))
 = \Ad (v) ({\overline{\ep}}_{1} (f))
\end{equation}
for all $f \in L^{\infty} ([0, 1])$,
and
\begin{equation}\label{Eq:AfTd3}
{\overline{\alpha}} (v) = e^{ - 2 \pi i / q} v.
\end{equation}
Thus ${\overline{\alpha}}^q = \Ad (v)$.
Then $N_q$ is the weak operator closure
of the image of $A_q$ under the Gelfand-Naimark-Segal
representation coming from the unique tracial state on $A_q$
(see Remark~\ref{R:2.2}),
and ${\overline{\alpha}} \in \Aut (N_q)$
is an extension of the automorphism $\alpha$
defined in (\ref{Eq:Af1}), (\ref{Eq:Af2}), and~(\ref{Eq:Af3}).
Identify $B_q$ with $C^* (g, u_0, u_1, u_2, \ldots )$
as in Lemma \ref{UHFmodel}(\ref{I_modelBq}).
Let $\om$ be the unique tracial state on $B_q$,
and let $R_0$ be the weak operator closure $\pi_{\om} (B_q)''$.
Then $R_0$ is isomorphic to the hyperfinite ${\mathrm{II}}_1$
factor~$R$.
Denote by ${\overline{\beta}}$ the extension to $R_0$ of the
automorphism $\beta$ of Lemma~\ref{L_3702_ExBt}.
Then, with $g \in B_q \subseteq \pi_{\om} (B_q)''$ being as in
Lemma~\ref{UHFmodel},
from Lemma \ref{L_3702_ExBt}(\ref{I_2816_btq}) and~(\ref{I_2816_bt_g}) we get
$\overline{\beta}^q = \Ad (g)$ and $\overline{\beta} (g) = e^{2 \pi i/q} g$.
It is well known, and easy to see, that these relations imply that
${\overline{\beta}}$ has period $q$ in ${\operatorname{Out}} (R_0)$.

Let $w$ be as in Lemma~\ref{R:2.5}.
The automorphism
${\overline{\gamma}}
 = \Ad (w) \circ
   \big( {\overline{\alpha}} \otimes {\overline{\beta}} \big)$
generates an action, which we also call ${\overline{\gamma}}$,
of $\Z_q$ on $N_q {\overline{\otimes}} R_0$.
Since ${\overline{\beta}}$ has period $q$ in ${\operatorname{Out}} (R_0)$,
Corollary 1.14 in~\cite{Kll} (or Theorem 13.1.16 in \cite{KR})
implies that ${\overline{\gm}}^k$ is outer for $k = 1, 2, \ldots, q - 1$.
By Proposition 13.1.5(ii) of~\cite{KR}, the crossed product
\begin{equation}\label{vNAcrossedproduct}
T_q = (N_q {\overline{\otimes}} R_0) \rtimes_{{\overline{\gamma}}} \Z_q
\end{equation}
is a factor of type~${\mathrm{II}}_1$.

\begin{remark}\label{hypercentral}
Let $G$ be a discrete group containing a nonabelian
free group and such that its von Neumann algebra $\mathcal{L} (G)$
is a factor.
Let $R$ be the hyperfinite $\mathrm{II}_1$ factor.
By the proof of Proposition 3.5 in \cite{Viola},
any central sequence in $\mathcal{L} (G) \overline{\otimes} R$ has the form
$(1 \otimes x_n)_{n \in \N}+ (y_n)_{n \in \N}$
for a central sequence $(x_n)_{n \in \N}$  in~$R$
and a sequence
$(y_n)_{n \in \N}$ in $\mathcal{L} (G) \overline{\otimes} R$
such that $\displaystyle{\lim_{n \to \infty} \| y_n \|_2 = 0}$.
Moreover, $\mathcal{L} (G) \overline{\otimes} R$
has no nontrivial hypercentral sequences, that is,
any hypercentral sequence in $\mathcal{L} (G) \overline{\otimes} R$ has the
form $(\lambda_n \cdot 1)_{n \in \N} + (y_n)_{n \in \N}$
for a sequence  $(\lambda_n)_{n \in \N}$ in
$\mathbb{C}$ and a sequence
$(y_n)_{n \in \N}$ in $\mathcal{L} (G) \overline{\otimes} R$
such that $\displaystyle{\lim_{n \to \infty} \| y_n \|_2 = 0}$.
\end{remark}

Using the Connes exact sequence in (\ref{Connessequence}),
we now compute the Connes invariant of $T_q$.
The argument follows Section 5 in \cite{Viola}, with suitable changes.
We reproduce it here for the convenience of the reader.

\begin{proposition}\label{ConnesinvariantM}
Let $q$ be an odd prime and let
$T_q
 = (N_q {\overline{\otimes}} R_0) \rtimes_{{\overline{\gamma}}} \Z_q$
be the $\mathrm{II}_1$ factor defined in (\ref{vNAcrossedproduct}).
Then $\chi (T_q) \cong \mathbb{Z}_{q^2}$.
Moreover, the unique subgroup
of order~$q$ in $\chi (T_q)$ is the image of the action
$\sm \colon {\widehat{\Z}_q} \to \Aut (T_q)$
obtained as the dual action on
$T_q =
  (N_q {\overline{\otimes}} R_0) \rtimes_{{\overline{\gamma}}} \Z_q$.
\end{proposition}

\begin{proof}
Denote by ${\mathbb{F}}_q$ the free group on $q$~generators.
For the first part of the proof, we use the facts and notation in the
discussion after Lemma~\ref{L_3705_Q}.
By abuse of notation, we write $\Z_q$ for the image of $\Z_q$
in $\Aut (N_q {\overline{\otimes}} R_0)$ under~${\overline{\gamma}}$.
Since $N_q \cong {\mathcal{L}} ({\mathbb{F}}_q \star {\mathbb{Z}}_q)$
is full by Lemma 3.2 in \cite{Viola} and
$\overline{\alpha} \notin \Inn (N_q)$, Corollary 3.3 in
\cite{Connes4} implies that
$\mathbb{Z}_q \cap \overline{\Inn (N_q {\overline{\otimes}} R_0)}
 = \{ 1 \}$.
Moreover, by Remark \ref{hypercentral}, the $\mathrm{II}_1$
factor $N_q {\overline{\otimes}} R_0$ has no nontrivial
hypercentral sequences.

To compute the Connes invariant
of $T_q$, we first compute the subgroups
$K^{\perp}$ and $L$ introduced after Remark~\ref{R:2.11}.

By Remark \ref{hypercentral}, any central sequence in
$N_q {\overline{\otimes}} R_0$ has the form
$(1 \otimes x_n)_{n \in \N} + (y_n)_{n \in \N}$
for a central sequence $(x_n)_{n \in \N}$ in~$R_0$ and
a sequence $(y_n)_{n \in \N}$ in $N_q {\overline{\otimes}} R_0$
such that $\displaystyle{\lim_{n \to \infty} \| y_n \|_2 = 0}$.
Since the trace on $N_q {\overline{\otimes}} R_0$ is unique,
it is $\gamma$-invariant, so also
$\displaystyle{\lim_{n \to \infty} \| \overline{\gamma} (y_n) \|_2 = 0}$.
Therefore $\overline{\gamma} \in \Ct (N_q {\overline{\otimes}} R_0)$
if and only if $\overline{\beta} \in \Ct (R_0)$.
Since $\Ct (R_0) = \Inn (R_0)$ by Theorem 3.2(1) in~\cite{Connes3},
and $\overline{\beta}$ is outer, it follows that
$K = \mathbb{Z}_q \cap \Ct (N_q {\overline{\otimes}} R_0)$ is trivial.
Thus $K^{\perp} \cong \mathbb{Z}_q$.

We next compute~$L$.
Using the notation of Lemma~\ref{UHFmodel},
we have $R_0 = \{g, u_0, u_1, u_2, \ldots \}''$ and
$\overline{\beta} = \Ad (u_{-1})$.
Let
$\xi \colon \Aut (N_q {\overline{\otimes}} R_0) \to
       \Out (N_q {\overline{\otimes}} R_0)$
denote the quotient map. We claim that $L \cong \mathbb{Z}_q$
and that a generator of $L$ is given by
\[
\mu = \xi (\id \otimes [\Ad (u_0^*) \circ \overline{\beta}]).
\]

We prove the claim.
Using $(\alpha \otimes \beta) (w) = w$
(by Lemma \ref{R:2.5}(\ref{I_2813_R_2_5_Fix})) at the first step,
\[
\begin{split}
\overline{\gamma} \,^{-1}
   \circ (\id \otimes [\Ad (u_0^*) \circ \overline{\beta}])
& = \Ad (w^*) \circ (\overline{\alpha} \,^{-1}
         \otimes \overline{\beta} \,^{-1})
         \circ (\id \otimes [\Ad (u_0^*) \circ \overline{\beta}])
\\
& = \Ad (w^* (1 \otimes u_0^*)) \circ (\overline{\alpha}  \,^{-1}\otimes \id).
\end{split}
\]
By Remark~\ref{hypercentral}, this automorphism
is in $\Ct (N_q {\overline{\otimes}} R_0)$, so that
\[
\id \otimes [\Ad (u_0^*) \circ \overline{\beta}]\in \mathbb{Z}_q \vee
       \Ct (N_q {\overline{\otimes}} R_0).
\]
For $n \in \N$ define $y_n = u_0 u_1^* u_2 u_3^* \cdots u_{2 n}$,
which is a unitary in~$R_0$.
Using Lemma \ref{UHFmodel}(\ref{I_2811_vkg_sm}) and~(\ref{I_2811_vkg_lg}),
we get $y_n g y_n^* = e^{2 \pi i / q} g$.
Using Lemma \ref{UHFmodel}(\ref{I_modelB_vk_vkp1}) and~(\ref{I_modelB_ujvk}),
we get
\[
y_n u_k y_n^*
 = \begin{cases}
   e^{- 2 \pi i / q^2} u_k  & \hspace*{1em} k = 0
        \\
   u_k & \hspace*{1em} k = 1, 2, \ldots, 2 n - 1
       \\
   e^{2 \pi i / q^2} u_k & \hspace*{1em} k = 2 n
        \\
   e^{- 2 \pi i / q^2} u_k & \hspace*{1em} k = 2 n + 1
       \\
   u_k & \hspace*{1em} k = 2 n + 2, \, 2 n + 3, \ldots.
\end{cases}
\]
(The calculation for $k = 1, 2, \ldots, 2 n - 1$ depends on
the parity of~$k$.)
Comparing these formulas with Lemma~\ref{L_3702_ExBt},
we see that $\displaystyle{\lim_{n \to \I} y_n x y_n^* = \bt (x)}$
for all $x \in \{ g, u_0, u_1, u_2, \ldots \}$,
and hence for all $x \in B_q$.
Therefore
$\displaystyle{\lim_{n \to \infty}
    \| \Ad (y_n) (x) - \overline{\beta} (x) \|_2 = 0}$
for all $x \in R_0$.
Also, Lemma \ref{L_3702_ExBt}(\ref{I_2816_u0}) and~(\ref{I_2816_oth_un})
imply $\overline{\beta} (y_n) = e^{-2 \pi i / q^2} y_n$ for $n \in \N$.
Set $z_n = u_0^* y_n$.
Then
\[
\id \otimes [\Ad (u_0^*) \circ \overline{\beta}]
 = \lim_{n \to \infty} \Ad (1 \otimes z_n)
\andeqn
\overline{\beta} (z_n) = z_n.
\]
It follows (recalling $H$ from~(\ref{Eq_2815_H})) that
\[
\id \otimes [\Ad (u_0^*) \circ \overline{\beta}]
 \in (\mathbb{Z}_q \vee
       \Ct (N_q {\overline{\otimes}} R_0)) \cap \overline{H}.
\]
So $\mu = \xi (\id \otimes [\Ad (u_0^*) \circ \overline{\beta}]) \in L$.

To finish the proof of the claim, since $\mu$ has order~$q$,
we show that $\mu$ generates $L$.
Given an automorphism
$\varphi \in
 (\mathbb{Z}_q \vee \Ct (N_q {\overline{\otimes}} R_0))
         \cap \overline{H}$,
there exists $k \in \{0, 1, \ldots , q - 1 \}$ such that
$\varphi \circ \overline{\gamma}^k$ is centrally trivial.
By the same argument as used to prove Proposition~3.6 in \cite{Viola},
there are $\nu \in \Aut (N_q)$
and a unitary $z \in N_q {\overline{\otimes}} R_0$ such that
$\varphi \circ \overline{\gamma}^k = \Ad (z) \circ (\nu \otimes \id)$.
Thus, with $x = z (\nu \otimes \id) ((w^*)^k)$, we have
\[
\varphi
 = \Ad (z) \circ (\nu \otimes \id)
   \circ \Ad (w^*)^k
   \circ (\overline{\alpha} \,^{-k}\otimes \overline{\beta} \,^{-k})
 = \Ad (x) \circ ([\nu \circ \overline{\alpha} \,^{-k}]
              \otimes \overline{\beta} \,^{-k}).
\]
Since
$\varphi
 \in \overline{H}
 \subseteq \overline{\Inn (N_q {\overline{\otimes}} R_0)}$
and $N_q$ is full, Corollary 3.3 in \cite{Connes4} implies that
$\nu \circ \overline{\alpha} \,^{-k} \in \Inn (N_q)$, so that
$\nu \circ \overline{\alpha} \,^{-k} = \Ad (c)$
for some unitary $c \in N_q$.
Thus
$\varphi
= \Ad (x (c \otimes 1)) \circ (\id \otimes \overline{\beta} \,^{-k})$
differs from a power of
$\id \otimes [\Ad (u_0^*) \circ \overline{\beta}]$
only by an inner automorphism.
This concludes the proof of the claim.

Next consider the Connes short exact sequence~(\ref{Connessequence}):
\begin{equation*}
\{ 1 \} \longrightarrow \mathbb{Z}_q
        \overset{\partial}{\longrightarrow} \chi (T_q)
        \overset{\Pi}{\longrightarrow} \mathbb{Z}_q
        \longrightarrow \{ 1 \}.
\end{equation*}
Taking
$\mu = \xi (\id \otimes [\Ad (u_0^*) \circ \overline{\beta}])$,
let $\et \in \Ct (T_q) \cap \overline{\Inn (T_q)}$
be the automorphism defined in (\ref{betamu}).
Let $\widetilde{\xi}\colon\Aut (T_q) \to \Out (T_q)$ 
denote the quotient map. Since the only possibilities for $\chi (T_q)$ are
$\mathbb{Z}_{q^2}$ and $\mathbb{Z}_q \oplus \mathbb{Z}_q$,
to complete the proof it is enough to show that $\widetilde{\xi} (\et)$
does not have order $q$, that is, $\et^q \notin \Inn (T_q)$.

By Lemma \ref{L_3702_ExBt}(\ref{I_2816_bt_g}) and~(\ref{I_2816_u0}), we have
\[
(\overline{\alpha} \otimes \overline{\beta}) (1 \otimes (u_0^*)^q g)
= e^{4 \pi i/q} (1 \otimes (u_0^*)^q g).
\]
Lemma \ref{R:2.5}(\ref{I_2813_R_2_5_Ct}) gives
$\Ad (w) (1 \otimes (u_0^*)^q g) = 1 \otimes (u_0^*)^q g$.
Hence,
$\overline{\gamma} (1 \otimes (u_0^*)^q g)
  = e^{4 \pi i / q} (1 \otimes (u_0^*)^q g)$.
Let $y$ be the standard unitary (of order $q$) implementing
the generating automorphism
$\Ad (w) \circ
 \big( {\overline{\alpha}} \otimes {\overline{\beta}} \big)$
of the action ${\overline{\gamma}}$ in the crossed product
$T_q =
 (N_q {\overline{\otimes}} R_0) \rtimes_{{\overline{\gamma}}} \Z_q$,
so that
$\Ad (w) \circ
   \big( {\overline{\alpha}} \otimes {\overline{\beta}} \big)
 = \Ad (y)$.
Then
\[
\Ad (1 \otimes (u_0^*)^q g) (y) = e^{- 4 \pi i / q} y.
\]
Recall from Lemma \ref{L_3702_ExBt}(\ref{I_2816_u0}) and~(\ref{I_2816_btq})
that $\overline{\beta} (u_0) = e^{-2 \pi i/q^2} u_0$
(which implies that $\overline{\beta}$ commutes with $\Ad (u_0^*)$)
and $\overline{\beta}^q = \Ad (g)$.
For $a_0, a_1, \ldots, a_{q - 1} \in N_q {\overline{\otimes}} R_0$,
it follows that
\begin{align*}
\et^q \left( \sum_{k = 0}^{q - 1} a_k y^k \right)
& = \sum_{k = 0}^{q - 1}
    (1 \otimes [\Ad (u_0^*)^q \circ \overline{\beta}^q]) (a_k) y^k
  = \sum_{k = 0}^{q - 1} (1 \otimes \Ad ((u_0^*)^q g)) (a_k)y^k
\\
& = \Ad (1 \otimes (u_0^*)^q g)
   \left( \sum_{k = 0}^{q - 1} e^{4 \pi i k / q} a_k y^k \right).
\end{align*}
Since $q$ is odd, we conclude that, up to an inner automorphism,
$\et^q$ is the dual action at a nontrivial element of the dual group.
Therefore $\et^q$ is outer.
\end{proof}

Using the Connes invariant, we can now show that the
\ca{} $D_q$ is not isomorphic to its opposite algebra
whenever $q$ is an odd prime such that $- 1$ is not a
square mod~$q$.

\begin{proposition}\label{P_3703_vNNotIsoOpp}
Let $q$ be any odd prime
such that $- 1$ is not a square mod~$q$.
Let $D_q = C_q \rtimes_{\gamma} {\mathbb{Z}}_q$
be the \ca{}
of Definition~\ref{D:2.4}.
Let $\ta$ be the unique tracial state on~$D_q$
(Proposition~\ref{pcaseproposition}),
and let $\pi_{\tau}$ be the Gelfand-Naimark-Segal representation
associated to~$\ta$.
Then the von Neumann algebra $\pi_{\tau} (D_q)''$
is not isomorphic to its opposite algebra.
\end{proposition}

\begin{proof}
We claim that $T_q \cong \pi_{\tau} (D_q)''$.
Let $\sm$ be the unique tracial state on~$A_q$ (Remark~\ref{R:2.2}),
and let $\om$ be the unique tracial state on the UHF algebra~$B_q$.
We have an obvious map
$C_q = A_q \otimes B_q \to N_q {\overline{\otimes}} R_0$,
which intertwines $\gm$ and ${\overline{\gm}}$.
Lemma~\ref{L_3705_Q} shows that this map induces
an isomorphism
$\pi_{\sm \otimes \om} (C_q)'' \cong N_q {\overline{\otimes}} R_0$.
Since the group is finite,
taking crossed products by~$\Z_q$
gives an isomorphism $T_q \cong \pi_{\tau} (D_q)''$, as claimed.

To show that $T_q$ is not isomorphic to its
opposite algebra, we give a recipe which starts with a factor~$P$,
just given as a factor of type~${\mathrm{II}}_1$
with certain properties
(see (\ref{InvP1}), (\ref{InvP2}), (\ref{InvP3}), and~(\ref{InvP4})
below),
and produces a subset $S_q (P)$ of $\Z_q$,
which we identify with $\{ 0, 1, \ldots, q - 1 \}$.
The important point is that this recipe does not depend
on knowing any particular element, automorphism, etc.\  of~$P$.
That is, if we start with some other factor of type~${\mathrm{II}}_1$
which is isomorphic to~$P$,
then we get the same subset of $\{ 0, 1, \ldots, q - 1 \}$,
regardless of the choice of isomorphism.
When $- 1$ is not a square mod~$q$,
we will show that the recipe also applies to $P^{\mathrm{op}}$
and gives a different subset,
from which it will follow that
$T_q^{\mathrm{op}} \not\cong T_q$.

We describe the construction first,
postponing the proofs that the steps can be carried out
and the result is independent of the choices made.
Let $P$ be a factor of type ${\mathrm{II}}_1$ with separable predual.
Let $\chi (P)$ denote the Connes invariant of~$P$
as in Definition~\ref{D:2.10},
and assume that $P$ satisfies the following properties:
\begin{enumerate}
\item\label{InvP1}
$\ch (P) \cong \Z_{q^2}$.
\item\label{InvP2}
The unique subgroup of $\ch (P)$ of order~$q$
is the image of a subgroup
(not necessarily unique) of $\Aut (P)$ isomorphic to~$\Z_q$.
\item\label{InvP3}
Let $\rh \colon \Z_q \to \Aut (P)$ come from a choice
of the subgroup and isomorphism in~(\ref{InvP2}).
Form the crossed product $P \rtimes_{\rh} \Z_q$,
and let
${\widehat{\rh}} \colon
  {\widehat{\Z_q}} \to \Aut (P \rtimes_{\rh} \Z_q)$
be the dual action.
Then for every nontrivial element
$l \in {\widehat{\Z_q}}$,
the automorphism ${\widehat{\rh}}_l \in \Aut (P \rtimes_{\rh} \Z_q)$
has a factorization $\ph \circ \ps$, in which $\ph$ is
an approximately inner automorphism and $\ps$
is a centrally trivial automorphism.
\item\label{InvP4}
For any nontrivial element
$l \in {\widehat{\Z_q}}$ and any factorization
${\widehat{\rh}}_l = \ph \circ \ps$ as in~(\ref{InvP3}),
there is a unitary $z \in P \rtimes_{\rh} \Z_q$ such that
$\ps^q = \Ad (z)$,
and there is $k \in \{ 0, 1, \ldots, q - 1 \}$
such that $\ps (z) = e^{2 \pi i k / q} z$.
(See the obstruction to lifting of Definition~\ref{D:2.12}.)
\end{enumerate}

For a type ${\mathrm{II}}_1$~factor $P$ which satisfies
(\ref{InvP1}), (\ref{InvP2}), (\ref{InvP3}), and~(\ref{InvP4}),
we take $S_q (P)$ to be
the set of all values of~$k \in \{ 0, 1, \ldots, q - 1 \}$
which appear in~(\ref{InvP4})
for any choice of the action $\rh \colon \Z_q \to \Aut (P)$,
any nontrivial element $l \in {\widehat{\Z_q}}$,
and any choice of the factorization
${\widehat{\rh}}_l = \ph \circ \ps$ as in~(\ref{InvP3}).
We think of $S_q (P)$ as a subset of $\Z_q$ in the obvious way.

We claim that the crossed product $P \rtimes_{\rh} \Z_q$
is uniquely determined up to isomorphism
and the dual action
${\widehat{\rh}}
  \colon {\widehat{\Z_q}} \to \Aut (P \rtimes_{\rh} \Z_q)$
is uniquely determined up to conjugacy and automorphisms of~$\Z_q$.
There are two ambiguities in the choice of~$\rh$.
If we change the isomorphism of
$\Z_q$ with the subgroup of $\ch (P)$ of order~$q$,
we are modifying $\rh$ by an automorphism of $\Z_q$.
The crossed product $M \rtimes_{\rh} \Z_q$ is the same,
and the dual action is modified by the corresponding
automorphism of ${\widehat{\Z_q}}$.
Now suppose that we fix an isomorphism of $\Z_q$
with the subgroup of $\ch (P)$ of order~$q$,
but choose a different lift ${\overline{\rh}}$
to a \hm\   $\Z_q \to \Aut (P)$.
Then Lemma~\ref{L:Cocycle}
implies that the crossed products are isomorphic
and the dual actions are conjugate.
This proves the claim.

Since if the dual action changes by
conjugation by an automorphism, the automorphisms
in the decomposition of~(\ref{InvP3}) also change
by conjugation by an automorphism, and the obstruction
to lifting is unchanged by conjugation,
it follows that changing
the dual action by conjugation leaves $S_q (P)$ invariant.
This shows that $S_q (P)$ can be computed
by fixing a particular choice of
$\rh \colon {\mathbb{Z}}_q \to {\operatorname{Aut}} (P)$.

Next we check that
if $P$ satisfies
(\ref{InvP1}), (\ref{InvP2}), (\ref{InvP3}), and~(\ref{InvP4}),
then so does $P^{\mathrm{op}}$.
For this purpose,
we use the von Neumann algebra $P^{\mathrm{c}}$
described in Definition~\ref{D:2.6}, which is isomorphic to
$P^{\mathrm{op}}$ by Remark~\ref{R:2.7}.
Scalar multiplication enters in the definition of $S_q (P)$
in only two places.
The first is the definition of the dual action
${\widehat{\rh}} \colon
 {\widehat{\Z_q}} \to \Aut (P \rtimes_{\rh} \Z_q)$.
However, the change is easily undone by applying
the automorphism $l \mapsto - l$ of ${\widehat{\Z_q}}$.
The other place is in the definition of the obstruction to lifting.
So $P^{\mathrm{c}}$ satisfies
the conditions
(\ref{InvP1}), (\ref{InvP2}), (\ref{InvP3}), and~(\ref{InvP4}),
and we get
\begin{equation}\label{Eq:ObstrOfPOp}
S_q ( P^{\mathrm{c}} ) = \{ - l \colon l \in S_q (P) \},
\end{equation}
where we are treating $S_q (\cdot)$ as a subset of~$\Z_q$.

In the rest of the proof,
we show that the ${\mathrm{II}}_1$~factor
$T_q
 = (N_q {\overline{\otimes}} R_0) \rtimes_{{\overline{\gamma}}} \Z_q$
satisfies
(\ref{InvP1}), (\ref{InvP2}), (\ref{InvP3}), and~(\ref{InvP4}),
and that moreover $S_q (T_q)$ can be computed
using, for each nontrivial element $l \in {\widehat{\Z_q}}$, just
one choice of the factorization
${\widehat{\rh}}_l = \ph \circ \ps$ in~(\ref{InvP3})
and one choice of the unitary $z$ in~(\ref{InvP4}).
We then finish by computing $S_q (T_q )$.

By Proposition \ref{ConnesinvariantM} we have
$\chi ( T_q) \cong \mathbb{Z}_{q^2}$, and
the unique subgroup
of order~$q$ in $\chi (T_q)$ is the image of the action
$\sm \colon {\widehat{\Z}_q} \to \Aut (T_q)$
obtained as the dual action on
$T_q =
  (N_q {\overline{\otimes}} R_0) \rtimes_{{\overline{\gamma}}} \Z_q$.
Thus, there exists at least one choice for~$\rh$,
namely $\sm$ composed with some isomorphism
$\Z_q \to {\widehat{\Z_q}}$.
Therefore $T_q$ satisfies property~(\ref{InvP2}).

By Takesaki's duality theory (see Theorem~4.5 of~\cite{Take}),
with $\lambda (g)$ denoting the left regular representation of
${\mathbb{Z}}_q$ on $l^2 ({\mathbb{Z}}_q)$,
there is an isomorphism
\[
T_q \rtimes_{\rh} \mathbb{Z}_q
 \cong (N_q {\overline{\otimes}} R_0) \otimes B (l^2 ({\mathbb{Z}}_q))
\]
which identifies $g \mapsto {\widehat{\rh}}_g$ with the tensor
product $g \mapsto {\overline{\gamma}}_g \otimes \Ad (\lambda (g)^*)$.

Now let $l \in {\widehat{\Z_q}}$.
We claim that ${\widehat{\rh}}_l$
can be written as $\ph \circ \ps$ for
an approximately inner automorphism $\ph$
and a centrally trivial automorphism~$\ps$,
and that this factorization is unique up to inner automorphisms.
This will imply property~(\ref{InvP3}).
We first consider uniqueness,
which is equivalent to showing that every automorphism
which is both approximately inner and centrally trivial is in fact
inner.
Since $N_q$ is full by Lemma 3.2 n \cite{Viola},
the decomposition described in the proof of Lemma~3.6 of~\cite{Viola}
can be used to show that every automorphism
of $N_q {\overline{\otimes}} R_0$
which is both approximately inner and centrally trivial is in fact
inner.
Uniqueness now follows because
$(N_q {\overline{\otimes}} R_0) \otimes B (l^2 ({\mathbb{Z}}_q))
 \cong N_q {\overline{\otimes}} R_0$.

For existence, since the approximately inner automorphisms
are a normal subgroup
of~$\Aut \big( T_q \rtimes_{\rh} \mathbb{Z}_q \big)$,
it suffices to take $l$ to be the standard generator
of~${\widehat{\Z_q}}$.
Equivalently,
consider ${\overline{\gamma}} \otimes \Ad (\lambda (1)^*)$.
We will take
\[
\ph = \big( \Ad (w)
        \circ \big( \id_{N_q} \otimes {\overline{\beta}} \big) \big)
      \otimes \Ad (\lambda (1)^*)
\andeqn
\ps = ({\overline{\alpha}} \otimes \id_{R_0})
      \otimes \id_{B (l^2 ({\mathbb{Z}}_q))}.
\]
It is clear that
${\overline{\gamma}} \otimes \Ad (\lambda (1)^*) = \ph \circ \ps$.
The automorphism $\ph$ is approximately inner because,
by construction, ${\overline{\beta}}$ is approximately inner.
(In fact, by Theorem XIV.2.16 of~\cite{Tk3}, every
automorphism of $R$ is approximately inner.)
To see that
$\ps$ is centrally trivial, we observe that by
Remark \ref{hypercentral} every central sequence in
$N_q {\overline{\otimes}} R_0$ has the form
$(1 \otimes x_n)_{n = 1}^{\infty} + (y_n)_{n = 1}^{\infty}$
for a central sequence
$(x_n)_{n = 1}^{\infty}$ in~$R_0$ and
a sequence $(y_n)_{n = 1}^{\infty}$ in $N_q {\overline{\otimes}} R_0$
such that $\displaystyle{\lim_{n \to \infty} \| y_n \|_2 = 0}$.
Since the trace on $N_q {\overline{\otimes}} R_0$ is unique,
it is $\ps$-invariant, so also
$\displaystyle{\lim_{n \to \infty} \| \ps (y_n) \|_2 = 0}$,
which implies that $\ps$ is centrally trivial.
This proves the claim.

The obstruction to lifting for $\ps$
(as in property~(\ref{InvP4})) is independent
of the choice of the unitary $z$ implementing $\ps^q$
because $T_q \rtimes_{\sm} {\widehat{{\mathbb{Z}}_q}}$
is a factor.
By the proof of Proposition~1.4
of~\cite{Connes3} it is unchanged if $\ps$ is replaced
by $\Ad (y) \circ \ps$ for any unitary
$y\in T_q \rtimes_{\sm} {\widehat{{\mathbb{Z}}_q}}$.
The centrally trivial factor in the decomposition
of any automorphism of
$T_q \rtimes_{\sm} {\widehat{{\mathbb{Z}}_q}}$,
in particular, of $({\overline{\gamma}} \otimes \Ad (\lambda (1)^*))^l$,
is determined up to inner automorphisms.
Since, moreover, $\ph$ and $\ps$ commute up to an inner automorphism,
it follows that we can compute $S_q (T_q)$
by simply computing the obstructions to lifting for all powers
$\ps^l$ for a fixed choice of $\ps$ and for
$l = 1, 2, \ldots, q - 1$.
We can take
\[
\ps = ({\overline{\alpha}} \otimes \id_{R_0})
 \otimes \id_{B (l^2 ({\mathbb{Z}}_q))},
\]
for which $z = v \otimes 1 \otimes 1$ has already been
shown to be a unitary
with $\ps^q = \Ad (z)$ and $\ps (z) = e^{ - 2 \pi i / q} z$.
Now one uses Equations
(\ref{Eq:AfTd1}), (\ref{Eq:AfTd2}), and~(\ref{Eq:AfTd3})
to check that
\[
(\ps^l)^q = \Ad (z^l)
\andeqn
\ps^l (z^l) = e^{ - 2 \pi i l^2 / q} z^l.
\]
Therefore
(identifying $\{0, 1, \ldots, q - 1 \}$ with $\Z_q$ in the usual way)
\[
S_q (T_q) = \big\{ - l^2 \colon l \in \Z_q \setminus \{ 0 \} \big\}.
\]
As observed in Equation~(\ref{Eq:ObstrOfPOp}) above,
$S_q (T_q^{\mathrm{op}})$
is then given by
\[
S_q (T_q^{\mathrm{op}}) =
\big\{ l^2 \colon l \in \Z_q \setminus \{ 0 \} \big\}.
\]
If $q$ is an odd prime such that $- 1$ is not a square mod~$q$,
we have
$S_q (T_q^{\mathrm{op}}) \neq S_q (T_q)$,
whence $T_q^{\mathrm{op}} \not\cong T_q$.
\end{proof}

\section{\ca{s} not isomorphic to their opposite
   algebras}\label{Sec_NotIsoOpp}

We use the results of Section~\ref{Sec_Main}
to produce a number of examples of simple exact \ca{s}
not isomorphic to their opposite algebras
and which satisfy
the Universal Coefficient Theorem.

\begin{lemma}\label{L_3705_AffHme}
Let $A$ and $B$ be unital \ca{s}
and assume that $B$ has a unique tracial state~$\tau$.
Then the map $\sm \mapsto \sigma \otimes \tau$
is an affine weak* homeomorphism from the tracial state
space $\T (A)$ of~$A$
to the tracial state space $\T (A \otimes B)$ of $A \otimes B$.
\end{lemma}

We can't use Proposition 6.12 in~\cite{CP},
because
(see Proposition 2.7 in~\cite{CP})
it assumes that $\T (A)$ is finite dimensional.

\begin{proof}
It is easy to check that the map $\sm \mapsto \sigma \otimes \tau$
is injective, by considering $(\sigma \otimes \tau) (a \otimes 1)$
for $a \in A$.

We prove surjectivity.
Let $\rh$ be a tracial state on $A \otimes B$.
Define a tracial state $\sm$ on $A$
by $\sm (a) = \rh (a \otimes 1)$ for $a \in A$.
We claim that $\sm \otimes \ta = \rh$.
It suffices to verify equality on $a \otimes b$
for $a \in A_{+}$ and $b \in B$.
So let $a \in A_{+}$.
Define a tracial positive linear functional
$\nu_a \colon B \to \C$
by $\nu_a (b) = \rh (a \otimes b)$
for $b \in B$.
By uniqueness of~$\ta$,
there is $\ld (a) \geq 0$
such that $\nu_a = \ld (a) \ta$.
Then
\[
\ld (a) = \nu_a (1) = \rh (a \otimes 1) = \sm (a).
\]
Thus for all $b \in B$,
we have
\[
(\sm \otimes \ta) (a \otimes b)
 = \ld (a) \ta (b)
 = \nu_a (b)
 = \rh (a \otimes b).
\]
This completes the proof of surjectivity,
and shows that the inverse map is given by
$\rh \mapsto \rh |_{A \otimes 1_B}$.

It is obvious that $\rh \mapsto \rh |_{A \otimes 1_B}$
is affine and is continuous for the weak* topologies.
Since both tracial state spaces are compact and Hausdorff,
it follows that $\sm \mapsto \sigma \otimes \tau$
is an affine homeomorphism.
\end{proof}

\begin{proposition}\label{P_3705_Main}
Let $q$ be any odd prime such that
$- 1$ is not a square mod~$q$.
Let $D_q$ be the \ca{}
of Definition~\ref{D:2.4}.
Let $E$ be a
simple separable unital nuclear stably finite \ca.
Then $E \otimes D_q$ is exact and
$E \otimes D_q \not\cong (E \otimes D_q)^{\mathrm{op}}$.
\end{proposition}

\begin{proof}
Exactness follows from Proposition 7.1(iii) of~\cite{Kr2}.

Let $\ta$ be the unique tracial state on~$D_q$
(Proposition~\ref{pcaseproposition}).
Let $R$ be the hyperfinite ${\mathrm{II}}_1$~factor.
We claim that
\begin{equation}\label{Eq_3705_RStab}
R {\overline{\otimes}} \pi_{\ta} (D_q)'' \cong \pi_{\ta} (D_q)''.
\end{equation}
To prove the claim,
let $\om$ be the unique tracial state on~$B_q$.
By Proposition~\ref{pcaseproposition},
there is an isomorphism
$\ph \colon D_q \to B_q \otimes D_q$.
Since $(\om \otimes \ta) \circ \ph$ is a tracial
state on~$D_q$, we have
$(\om \otimes \ta) \circ \ph = \ta$.
Therefore
$\pi_{\om \otimes \ta} (B_q \otimes D_q)'' \cong \pi_{\ta} (D_q)''$.
Since
$\pi_{\om \otimes \ta} (B_q \otimes D_q)''
  \cong \pi_{\om} (B_q)'' {\overline{\otimes}} \pi_{\ta} (D_q)''$
by Lemma~\ref{L_3705_Q},
and $\pi_{\om} (B_q)'' \cong R$,
the claim follows.

We now claim that we may assume that $E \otimes B_q \cong E$.
Indeed, $B_q \otimes D_q \cong D_q$
by Proposition~\ref{pcaseproposition},
so that $(E \otimes B_q) \otimes D_q \cong E \otimes D_q$.
Accordingly, we may assume that $E$ is
infinite dimensional.
By Corollary~9.14 of~\cite{HT},
there is a tracial state on~$E$.
So the Krein-Milman Theorem
provides an extreme tracial state on~$E$.

For any extreme tracial state $\sm$ on~$E$, we have
\begin{equation}\label{Eq_3705_IsoR}
\pi_{\sm} (E)'' \cong R.
\end{equation}
Using Lemma~\ref{L_3705_Q} at the first step,
(\ref{Eq_3705_IsoR}) at the second step,
and (\ref{Eq_3705_RStab}) at the third step, we then get
\begin{equation}\label{Eq_3705_Absb}
\pi_{\sm \otimes \ta} (E \otimes D_q)''
  \cong \pi_{\sm} (E)'' {\overline{\otimes}} \pi_{\ta} (D_q)''
  \cong R {\overline{\otimes}} \pi_{\ta} (D_q)''
  \cong \pi_{\ta} (D_q)''.
\end{equation}

Now suppose that there is an isomorphism
$\ps \colon (E \otimes D_q)^{\mathrm{op}} \to E \otimes D_q$.
Let $\sm$ be any extreme tracial state on~$E$.
It follows from Lemma~\ref{L_3705_AffHme}
that $\sm \otimes \ta$ is an extreme tracial state on
$E \otimes D_q$.
Therefore $(\sm \otimes \ta) \circ \ps$
is an extreme tracial state on
$(E \otimes D_q)^{\mathrm{op}}
 \cong E^{\mathrm{op}} \otimes D_q^{\mathrm{op}}$.
Lemma~\ref{L_3705_AffHme} now provides
an extreme tracial state $\rh$ on~$E^{\mathrm{op}}$
such that
$(\sm \otimes \ta) \circ \ps = \rh \otimes \ta^{\mathrm{op}}$.
The state $\rh^{\mathrm{op}}$ is clearly extreme.
Using (\ref{Eq_3705_Absb}) at the first step,
Lemma~\ref{L_3704_OppTr} at the fourth step,
Lemma~\ref{L_3705_Q} at the fifth step,
and (\ref{Eq_3705_Absb}) with $\rh^{\mathrm{op}}$ in place of~$\sm$
at the sixth step,
we therefore get
\begin{align*}
\pi_{\ta} (D_q)''
&\cong \pi_{\sm \otimes \tau} (E \otimes D_q)''
  \cong \pi_{(\sm \otimes \tau) \circ \ps}
      ((E \otimes D_q)^{\mathrm{op}})''
         \\
&= \pi_{\rh \otimes \ta^{\mathrm{op}}}
        ((E \otimes D_q)^{\mathrm{op}})''
  \cong [\pi_{\rh^{\mathrm{op}} \otimes \ta}
    (E \otimes D_q)'']^{\mathrm{op}}
     \\
&\cong [\pi_{\rh^{\mathrm{op}}} (E)''
      {\overline{\otimes}} \pi_{\ta} (D_q)'']^{\mathrm{op}}
 \cong [\pi_{\ta} (D_q)'']^{\mathrm{op}}.
\end{align*}
This contradicts Proposition~\ref{P_3703_vNNotIsoOpp}.
\end{proof}

We use Proposition~\ref{P_3705_Main}
to give many examples of simple separable exact \ca{s}
not isomorphic to their opposite algebras.
Many other variations are possible.
The ones we give are chosen to demonstrate the possibilities
of nontrivial~$K_1$,
of $K_0$ being the same as that of many different UHF~algebras,
of real rank one rather than zero,
and of having many tracial states.

\begin{theorem}\label{T_3705_Basic}
Let $q$ be an odd prime such that $- 1$ is not a square mod~$q$.
Then there exists a simple separable unital exact \ca~$A$
not isomorphic to its opposite algebra
which is approximately divisible and stably finite,
has stable rank one,
tensorially absorbs the $q^{\infty}$ UHF algebra
and the Jiang-Su algebra, and
has the property that traces determine the order on projections
over~$A$.
In addition, $A$ has the following properties:
\begin{enumerate}
\item\label{T_3705_Basic_Kth}
$K_0 (A) \cong {\mathbb{Z}} \big[ \frac{1}{q} \big]$
with $[1_A] \mapsto 1$
and
$K_0 (A)_{+} \mapsto {\mathbb{Z}} \big[ \frac{1}{q} \big] \cap [0, \infty)$.
\setcounter{Tmp2}{\value{enumi}}
\item\label{T_3705_Basic_K1}
$K_1 (A) = 0$.
\setcounter{TmpEnumi}{\value{enumi}}
\item\label{T_3705_Basic_Cu}
$W (A)
 \cong {\mathbb{Z}} \big[ \tfrac{1}{q} \big]_{+} \amalg (0, \infty)$.
\item\label{T_3705_Basic_RR}
$A$ has real rank zero.
\item\label{T_3705_Basic_T}
$A$ has a unique tracial state.
\item\label{T_3705_UCT}
$A$ satisfies the Universal Coefficient Theorem.
\end{enumerate}
\end{theorem}

\begin{proof}
Take $A$ to be the \ca~$D_q$ of Definition~\ref{D:2.4}.
Then $A \not\cong A^{\mathrm{op}}$ by
Proposition \ref{P_3703_vNNotIsoOpp}
(or equivalently by Proposition~\ref{P_3705_Main} with
$E = \C$).
All the other properties follow from
Proposition~\ref{pcaseproposition} and
Proposition~\ref{C:2.7}.
\end{proof}

\begin{theorem}\label{T_3705_UHF}
Let $q$ be
an odd prime such that $- 1$ is not a square mod~$q$.
Let $B$ be any UHF algebra whose ``supernatural number''
is divisible by arbitrarily large powers of~$q$.
Then there exists a \ca~$A$
as in Theorem~\ref{T_3705_Basic},
except that~(\ref{T_3705_Basic_Kth})
and~(\ref{T_3705_Basic_Cu})
are replaced by:
\begin{enumerate}
\item\label{T_3705_UHF_Kth}
$K_0 (A) \cong K_0 (B)$ as a scaled ordered group.
\setcounter{enumi}{\value{TmpEnumi}}
\item\label{T_3705_UHF_Cu}
$W (A) \cong K_0 (B)_{+} \amalg (0, \infty)$.
\end{enumerate}
\end{theorem}

\begin{proof}
Let $D_q$ be as in Definition~\ref{D:2.4}.
Take $A = B \otimes D_q$.
Then exactness of $A$ and $A \not\cong A^{\mathrm{op}}$
follow from Proposition~\ref{P_3705_Main}.
Since $D_q$
satisfies the Universal Coefficient Theorem, and $B$
belongs to the nuclear bootstrap category, $A$ satisfies the
Universal Coefficient Theorem.

The condition on $B$ implies that
$K_0 (B) \otimes_{\Z} {\mathbb{Z}} \big[ \tfrac{1}{q} \big]
  \cong K_0 (B)$.
Moreover,
${\operatorname{Tor}}_1^{\Z} \big( K_* (B), \,
                 {\mathbb{Z}} \big[ \tfrac{1}{q} \big] \big)$
is clearly zero.
Since $B$ is in the bootstrap class,
the K\"{u}nneth formula of~\cite{Sc2} gives
$K_0 (A) \cong K_0 (B)$
and $K_1 (A) = 0$.

It is obvious that $A$ is separable and unital.
Simplicity of $A$ follows from
simplicity of $B$ and $D_q$ and nuclearity of~$B$,
by the corollary on page~117 in~\cite{Take1}.
(We warn that this reference systematically
refers to tensor products as ``direct products''.)
Since $D_q$ has a unique tracial state
(by Proposition~\ref{pcaseproposition}),
Lemma~\ref{L_3705_AffHme} implies that
$A$ has a unique tracial state.
Combined with simplicity, this gives stable finiteness.
The algebra $A$ absorbs both the UHF algebra $B_q$ and $Z$ because
$D_q$ does
(by Proposition~\ref{pcaseproposition}).
Since $A$ is stably finite, $B_q$ is a UHF algebra,
and $B_q \otimes A \cong A$,
Corollary~6.6 of~\cite{Rordam1} implies that
${\operatorname{tsr}} (A) = 1$.
The algebra $D_q$ is approximately divisible
by Proposition~\ref{pcaseproposition}.
So $A = B \otimes D_q$ is approximately divisible.
Since $A$
is simple, approximately divisible, exact,
and has a unique tracial state,
it has real rank zero by Theorem~1.4(f) of~\cite{BKR}.
It follows from Proposition~2.6 of~\cite{PhV}
that the order on projections over $A$ is determined by traces.
The computation of $W (A)$ follows from the computation
of $K_0 (A)$ above,
$Z \otimes A \cong A$,
and Remark~\ref{R:2.15}.
\end{proof}

\begin{theorem}\label{T_3705_K1}
Let $q$ be
an odd prime
such that $- 1$ is not a square mod~$q$.
Let $G$ be any countable abelian group.
Then there exists a \ca~$A$
as in Theorem~\ref{T_3705_Basic},
except that~(\ref{T_3705_Basic_K1})
is replaced by:
\begin{enumerate}
\setcounter{enumi}{\value{Tmp2}}
\item\label{T_3705_K1_K1}
$K_1 (A) \cong G \otimes_{\Z} {\mathbb{Z}} \big[ \frac{1}{q} \big]$.
\end{enumerate}
\end{theorem}

\begin{proof}
Choose, using Theorem 4.20 of~\cite{EG},
a simple unital AH~algebra~$E$
with a unique tracial state,
such that $K_0 (E) \cong {\mathbb{Z}} \big[ \tfrac{1}{q} \big]$,
with $[1_E] \mapsto 1$,
and such that $K_1 (E) \cong G$.
Let $D_q$ be as in Definition~\ref{D:2.4}.
Take $A = E \otimes D_q$.
Using $E$ in place of~$B$,
proceed as in the proof of Theorem~\ref{T_3705_UHF}.
The only difference is in the computation of $K_* (A)$.
We have
\[
K_0 (E) \otimes_{\Z} {\mathbb{Z}} \big[ \tfrac{1}{q} \big]
  \cong {\mathbb{Z}} \big[ \tfrac{1}{q} \big]
               \otimes_{\Z} {\mathbb{Z}} \big[ \tfrac{1}{q} \big]
  \cong {\mathbb{Z}} \big[ \tfrac{1}{q} \big],
\]
and
\[
{\operatorname{Tor}}_1^{\Z} \big( K_0 (E), \,
                 {\mathbb{Z}} \big[ \tfrac{1}{q} \big] \big)
  = 0
\andeqn
{\operatorname{Tor}}_1^{\Z} \big( K_1 (E), \,
                 {\mathbb{Z}} \big[ \tfrac{1}{q} \big] \big)
  = 0.
\]
So the K\"{u}nneth formula of~\cite{Sc2}
implies that $K_* (A)$ is as claimed.
\end{proof}

\begin{theorem}\label{T_3705_Dt}
Let $q$ be  an odd prime
such that $- 1$ is not a square mod~$q$.
Let $\Dt$ be any Choquet simplex with more than one point.
Then there exists a \ca~$A$
as in Theorem~\ref{T_3705_Basic},
except that
(\ref{T_3705_Basic_Cu}),
(\ref{T_3705_Basic_RR}), and~(\ref{T_3705_Basic_T})
are replaced by:
\begin{enumerate}
\setcounter{enumi}{\value{TmpEnumi}}
\item\label{T_3705_Dt_Cu}
$W (A)
 \cong {\mathbb{Z}} \big[ \tfrac{1}{q} \big]_{+}
    \amalg {\operatorname{LAff}}_{\operatorname{b}} (\Dt)_{++}$.
\item\label{T_3705_Dt_RR}
$A$ has real rank one.
\item\label{T_3705_Dt_T}
$\T (A) \cong \Dt$.
\end{enumerate}
\end{theorem}

\begin{proof}
Using Theorem 3.9 of~\cite{Thn}, choose
a simple unital AI~algebra $E$
such that $K_0 (E) \cong {\mathbb{Z}} \big[ \tfrac{1}{q} \big]$,
with $[1_E] \mapsto 1$,
and $\T (E) \cong \Dt$.
Let $D_q$ be as in Definition~\ref{D:2.4}.
Take $A = E \otimes D_q$.
Using $E$ in place of~$B$,
proceed as in the proof of Theorem~\ref{T_3705_UHF}.
The differences are as follows.
Here,
since $D_q$ has a unique tracial state
(by Proposition~\ref{pcaseproposition}),
Lemma~\ref{L_3705_AffHme}
gives an affine homeomorphism from $\T (E) \cong \Dt$
to $\T (A)$.
The computation of $W (A)$ is as before,
but the answer is different because $\T (A) \cong \Dt$
instead of being a point.
Since there is only one state on the scaled ordered group $K_0 (A)$,
all tracial states must agree on all projections in~$A$.
Since $\Dt$ has more than one point,
the projections in~$A$ do not distinguish the tracial states.
So $A$ does not have real rank zero by Theorem 1.4(e)
in~\cite{BKR} and Theorem 5.11 in~\cite{Haa}.
However, we still get ${\operatorname{tsr}} (A) = 1$,
so $A$ has real rank at most~$1$ by Proposition~1.2 of~\cite{BP}.
\end{proof}

\begin{remark}\label{R_3705_Unctbl}
Each of Theorem~\ref{T_3705_UHF},
Theorem~\ref{T_3705_K1},
and Theorem~\ref{T_3705_Dt}
(separately)
gives uncountably many mutually nonisomorphic \ca{s}
satisfying the Universal Coefficient Theorem.
\end{remark}

\section{Open questions}\label{Sec_Probs}

\begin{question}\label{Q_3705_D5}
Let $q$ be an odd prime such that $-1$ is a square mod~$q$.
Is is still true that $D_q$,
as in Definition~\ref{D:2.4},
is not isomorphic to its opposite algebra?
\end{question}

The invariant we use,
the obstruction to lifting,
no longer distinguishes $D_q$ and $(D_q)^{\mathrm{op}}$,
but this does not mean that they are isomorphic.

Even if $D_q \cong (D_q)^{\mathrm{op}}$,
different methods might give a positive answer to the following
question.

\begin{question}\label{Q_3705_pIs5}
Let $q$ be an odd prime such that $-1$ is a square mod~$q$.
Does there exist a simple separable unital exact stably finite \ca~$A$
not isomorphic to its opposite algebra
such that $K_0 (A) \cong {\mathbb{Z}} \big[ \frac{1}{q} \big]$
and $K_1 (A) = 0$?
\end{question}

Of course,
we would really like to get all the other properties
in Theorem~\ref{T_3705_Basic} as well,
in particular, unique tracial state, real rank zero,
and $B_q \otimes A \cong A$.

\begin{question}\label{Q_3705_JS}
Does there exist
a simple separable unital exact stably finite \ca~$A$
not isomorphic to its opposite algebra
such that $K_0 (A) \cong {\mathbb{Z}}$, with $[1_A] \mapsto 1$,
and $K_1 (A) = 0$?
\end{question}

Such an algebra would have no nontrivial projections.

\begin{question}\label{Q_3705_PI}
Does there exist
a simple separable unital purely infinite \ca~$A$
not isomorphic to its opposite algebra?
\end{question}

\begin{question}\label{Q_2816_Nuc}
Does there exist a simple separable unital nuclear \ca{} $A$
not isomorphic to its opposite algebra?
\end{question}

Under an additional axiom of set theory,
nonseparable examples are known \cite{FrhHbg}.
By classification, a separable example can't both absorb the Jiang-Su algebra
tensorially and satisfy the Universal Coefficient Theorem.

\section*{Acknowledgments}

Some of this work was carried out during a
summer school held at the CRM in Bellaterra,
during a visit to the Fields Institute in Toronto, and
during a visit by the second author to the University of Oregon.
Both authors are grateful to the CRM and the Fields
Institute for their hospitality.
The second author would also like
to thank the University of Oregon for its hospitality.

\end{document}